\documentclass[11pt]{article}
\pdfoutput=1

\usepackage[T1]{fontenc}
\usepackage[utf8]{inputenc}
\usepackage{lmodern}
\usepackage{microtype}

\usepackage{amsmath,amssymb,amsthm,latexsym,mathtools}
\usepackage{graphicx}
\usepackage{xcolor}
\usepackage{booktabs}
\usepackage{nicematrix}
\usepackage[ruled,vlined]{algorithm2e}
\usepackage{enumitem}
\usepackage[left=1in,right=1in,top=0.8in,bottom=0.9in]{geometry}
\usepackage{changepage}
\usepackage{hyperref}
\usepackage{cleveref}

\usepackage[numbers,sort&compress]{natbib}

\newcommand{\hl}[1]{#1}
\newcommand{\highlighting}[1]{#1}
\newcommand{\subscript}[2]{#1#2}

\newcommand{\Prob}{\ensuremath{\mathbb{P}}}
\newcommand{\Ex}{\ensuremath{\mathbb{E}}}

\newcommand{\argmin}{\mathop{\rm arg\,min}}
\def\maxim{\mathop{\textup{maximize}}}
\def\minim{\mathop{\textup{minimize}}}
\def\st{\mathop{\textup{subject to}}}

\providecommand{\natexlab}[1]{#1}

\theoremstyle{plain}
\newtheorem{theorem}{Theorem}
\newtheorem{lemma}{Lemma}
\newtheorem{proposition}{Proposition}
\newtheorem{corollary}{Corollary}
\theoremstyle{definition}
\newtheorem{definition}{Definition}
\newtheorem{assumption}{Assumption}
\newtheorem{remark}{Remark}

\title{Multi-gear bandits, partial conservation laws, and indexability\thanks{This paper was presented by the author at the 32nd European Conference on Operational Research (EURO 2022), Espoo, Finland, 3--6 July 2022. An earlier preliminary draft was presented by the author at the 3rd International Conference on Performance Evaluation Methodologies and Tools (ValueTools 2008), Athens, Greece, 20--24 October 2008, where an extended abstract was published in its {proceedings}.}}
\author{Jos\'e Ni\~no-Mora\\
Department of Statistics, Universidad Carlos III de Madrid\\
28903 Getafe (Madrid), Spain\\
\texttt{jose.nino@uc3m.es}}
\date{} 

\begin{document}

\maketitle

\begin{abstract}
This paper considers what we propose to call \emph{\highlighting{multi-gear bandits}}, which are Markov decision processes modeling a generic dynamic and stochastic \emph{project} fueled by a single resource and which admit multiple actions representing gears of operation naturally ordered by their increasing resource consumption. 
The optimal operation of a multi-gear bandit aims to strike a balance between project performance costs or rewards and resource usage costs, which depend on the resource price. 
A computationally convenient and intuitive optimal solution is available when 
such a model is \emph{indexable}, meaning that its optimal policies are characterized by a \emph{dynamic allocation index} (DAI), a function of state--action pairs representing critical resource prices.
Motivated by the lack of general indexability conditions and efficient index-computing schemes, and focusing on the infinite-horizon finite-state and -action discounted case, we present a verification theorem ensuring that, if a model satisfies two proposed \emph{PCL-indexability conditions} with respect to a postulated family of structured policies, then it is indexable and such policies are optimal, with its DAI being given by a \emph{marginal productivity index} computed by a downshift adaptive-greedy algorithm in $A N$ steps, with $A+1$ actions and $N$ states. The DAI is further used as the basis of a new index policy for the \emph{multi-armed multi-gear bandit problem}.
\end{abstract}



\medskip
\noindent\textbf{Note:} Published in \emph{Mathematics} \textbf{10}, 2497 (2022). DOI:\ \url{https://doi.org/10.3390/math10142497}

\section{Introduction}
\label{s:intro}
There is a 
substantial literature
analyzing the \emph{indexability} of infinite-horizon discrete-time binary-action Markov decision processes (MDP), i.e., their optimal solution by \emph{index policies}, starting with the seminal work of Bellman in~\cite{bellman56}
on a Bernoulli bandit model.
Such MDPs can be interpreted as models of a generic dynamic and stochastic \emph{project}, which can be operated in a passive or an active mode. We propose to refer to such operating modes as \emph{gears}, to reflect their natural ordering by increasing activity level.

The model in~\cite{bellman56} had the property 
that, while the project is passive, its state  does not change, which corresponds to a  \emph{classic bandit} setting. 
It was later shown that, given a finite collection of classic bandits, one of which must be active at 
each time, the policy that maximizes the expected total discounted reward in such a \emph{multi-armed bandit problem} has a remarkably simple structure which overcomes the \emph{curse of dimensionality} of a standard \emph{dynamic programming} approach: it suffices to evaluate what Gittins and Jones~\cite{gijo74} called the \emph{dynamic allocation index} (DAI), later known as the \emph{Gittins index}, of each project, which is a function of its state, 
and then activate at each time a project with largest index. See, e.g., the seminal work  of Gittins and Jones~\cite{gijo74} and Gittins~\cite{gi79}, the monograph by Gittins~\cite{gi89}, and alternative proofs by Whittle~\cite{whittle80}, Weber~\cite{weber92}, and Bertsimas and Ni\~no-Mora~\cite{bnm96}.

The assumption in classic bandit models that passive projects do not change state was removed by Whittle in~\cite{whit88b}, introducing \emph{restless bandits}. 
That paper further introduced an index for restless bandits (the \emph{Whittle index}), which characterizes the optimal operation of a single restless project and provides a suboptimal heuristic policy for scheduling multiple such projects, in the so-called \emph{multi-armed restless bandit problem}.  The latter has huge modeling power but is computationally intractable, and Whittle's index policy has proven effective in an ever-increasing variety of models for multifarious applications. Thus, e.g., to name a few, scheduling multi-class make-to-stock queues~\cite{vewe96}, scheduling multi-class queues with finite buffers~\cite{nmqs06}, admission control and routing to parallel queues with reneging~\cite{nmnetcoop07}, obsolescence mitigation strategies~\cite{kuSa10}, sensor scheduling and dynamic channel selection~\cite{washbSchn08, liuZhao10, borkaretal18, yangLuo20}, group maintenance~\cite{abbuMakis19}, multi-target tracking with Kalman filter dynamics~\cite{lascala06, danceSi19}, scheduling multi-armed bandits with switching costs~\cite{nmijoc08} or switching delays~\cite{nmmath21}, the dynamic prioritization of medical treatments or interventions~\cite{ayeretal19, mateetal19}, and resource allocation with varying requests 
and with resources shared by multiple requests~\cite{fuetal22}. 

Yet, while the Gittins index is well defined for any classic bandit, the Whittle index 
only exists for some restless bandits, called \emph{indexable}. Whittle pointed out in~\cite{whit88b} the need of finding sufficient conditions for \emph{indexability} (the existence of \mbox{the index).}

While researchers have deployed a wide variety of ingenious ad hoc techniques for proving the indexability of particular restless bandit models and computing their Whittle 
indices, 
the author has developed over the last two decades in a series of papers a systematic approach to accomplish such goals.
Thus, Ref. \cite{nmaap01} introduced a framework for establishing \emph{both} the indexability of a general finite-state restless bandit \emph{and} the optimality of a postulated family of structured policies, based on the concept of \emph{partial conservation laws} (PCLs), also introduced there.
If project performance metrics satisfy so-called \emph{PCL-indexability conditions}, 
then the project's Whittle index can be efficiently computed in $N$ steps, where $N$ is the number of states, by an \emph{adaptive-greedy index algorithm}. 
This algorithm is an extension of the classic index-computing algorithm of Klimov~\cite{kl74} for computing the indices that characterize the optimal policy for scheduling a multi-class queue with feedback. Note that Klimov's algorithm was adapted in~\cite{bnm96} for computing the Gittins index, based on a framework of generalized conservation laws.

PCLs extend classical conservation laws in stochastic scheduling. See, e.g., the conservation laws 
in Coffman and Mitrani~\cite{coffMitrani80}, 
the \emph{strong conservation laws} in Shanthikumar and Yao~\cite{shanthiYao92}, and the \emph{generalized conservation laws} in the work of Bertsimas and Ni\~no-Mora~\cite{bnm96}.

The author further developed the PCL framework for analyzing the indexability of finite-state restless bandits in~\cite{nmmp02}, which introduced projects fueled by 
a generic resource with a general resource consumption function. 
The framework was extended to the countably infinite state space case in~\cite{nmmor06} and to the \emph{bias optimality criterion} in~\cite{nmqs06}. Such early work is surveyed in the discussion paper
~\cite{nmtop07}. The framework was then extended to projects with a continuous real state in~\cite{nmmor20}, motivated by sensor scheduling applications. As for the adaptive-greedy algorithm for the Whittle index, an efficient computational implementation was presented in~\cite{nmmath20}.

The extension of the concept of indexability  from two-gear bandits to multi-gear bandits was first outlined  by Weber~\cite{weber07} in the setting of an illustrative example given by a three-action queueing admission control model.
Ref.~\cite{weber07} further outlined an index-computing algorithm for such a three-action model extending the aforementioned adaptive-greedy algorithm for the Whittle index. Such an insightful outline, was, however, not theoretically supported. 

The author formalized and extended Weber's~\cite{weber07} outline to introduce in~\cite{nmvaluet08a} a general multi-armed multi-mode bandit problem with finite-state projects, motivated by a model of optimal dynamic power allocation to multiple
users sharing a wireless downlink communication channel subject to a peak energy constraint, proposing an index policy when individual projects are indexable. Ref.~\cite{nmvaluet08a} further outlined an extension of the PCL framework along with an index algorithm, yet without proofs or analyses.  

See also the recent work of Zayas-Cab\'an et al.~\cite{zayascetal19} on a finite-horizon multi-armed multi-gear bandit problem and that of Killian et al.~\cite{killianetal21} on multi-action bandits, yet without a focus on indexability.

A related strand of research is the work on the optimality of structured policies in MDP models, mostly focusing on the monotonicity of optimal actions on the state. See, e.g., Serfozo~\cite{serfozo76} and the book chapter by Heyman and Sobel (\cite{heymSob84ii} Ch.\ 8). 
In some models with a one-dimensional state, most notably those arising in queueing theory, researchers have established the optimality of policies given by multiple thresholds. Thus, e.g., 
Crabill~\cite{crabill72} shows that such policies are optimal for selecting the speed of a single server catering to a queue. Such a model is an example of what we call here a \emph{multi-gear bandit}. 
Similar results are obtained, e.g.,  in Sabeti~\cite{sabeti73}, Ata and Shneorson~\cite{ataShneorson06}, and Mayorga et al.~\cite{mayorga06}.
The methods proposed herein also serve to establish the optimality of postulated families of structured policies, different from  prevailing approaches, typically based on \emph{submodularity}.

Another related line of work is the computational complexity of solving discounted finite-state and -action MDPs with general-purpose algorithms, most notably the classical methods of \emph{value iteration}, \emph{policy iteration}, and \emph{linear optimization}. 
Such methods are not \emph{strongly polynomial} in that the number of iterations required to compute an optimal policy depends not only on the number of states and actions but also on other factors, most notably the discount factor. 
Thus, e.g., Ye~\cite{ye11}  showed that the number of iterations required by policy iteration is 
bounded by $O((N^2 A/(1-\beta)) \log(N/(1-\beta))$, which shows that policy iteration is strongly polynomial but only if the discount factor is fixed and hence not part of the input. Such a bound has been improved by Scherrer~\cite{scherrer16} to 
$O((N A/(1~-~\beta)) \log(1/(1-\beta)))$.
Otherwise, policy iteration is known to have exponential worst-case complexity. See Hollanders et al.~\cite{hollandersetal16} and references therein. 
In contrast, as we will see, the algorithm presented herein solves a \emph{multi-gear bandit} model with $N$ states and $A+1$ actions in precisely $A N$ steps and is hence a strongly polynomial algorithm.

In contrast with the aforementioned outlines in the earlier work~\cite{weber07,nmvaluet08a}, this paper presents a theoretically supported extension of  the PCL-based sufficient indexability conditions from two-gear bandits to multi-gear bandits, along with 
an intuitive and efficient algorithm for computing the model's index.

The main contribution is a \emph{verification theorem} (Theorem \ref{the:verthe}) which ensures that, if the performance metrics of a multi-gear bandit model under a postulated family of structured  policies satisfy two \emph{PCL-indexability} conditions, then \emph{both} the model is indexable \emph{and} such policies are optimal, with the model's index being computed by a \emph{downshift adaptive-greedy algorithm} in $A N$ steps as pointed out above.

The remainder of the paper is organized as follows. 
Section \ref{s:pfmr} describes the multi-gear bandit model and formulates the main result, the verification theorem for indexability.
\mbox{Sections \ref{s:lpreflpp}--\ref{s:pcl}} lay out the groundwork needed to prove the verification theorem.
Thus, \mbox{Section \ref{s:lpreflpp}} discusses the linear optimization reformulation of the relevant MDP model.
\mbox{Section \ref{s:dperfmetr}} presents the required
relations between project performance metrics.
\mbox{Section \ref{s:aoadsf}}  analyzes the output of the proposed index-computing algorithm.
\mbox{Section \ref{s:pcl}} presents the framework of partial conservation laws in the present setting.
Then, \mbox{Section \ref{s:proofotvthe}} draws on the above to present our proof of the verification theorem.
\mbox{Section \ref{s:bpMAMGBP}} applies the indexability property to provide a performance bound and a novel index policy for the multi-armed multi-gear bandit problem.
\mbox{Section \ref{s:sext}} 
outlines some extensions, in particular to the long-run average cost criterion \mbox{(Section \ref{s:elracc})}, to models with uncontrollable states \mbox{(Section \ref{s:mwuncst})} and to models with a countably infinite state space (Section \ref{s:mcisspa}). 
Finally, Section \ref{s:disc} concludes with a discussion of the results.


\section{Preliminaries and Formulation of the Main Result}
\label{s:pfmr}
\subsection{Multi-Gear Bandits}
\label{s:malpp}
We next describe a general MDP model for the optimal operation of a multi-gear dynamic and stochastic project, which we call the \emph{multi-gear bandit problem}. Consider a general discrete-time infinite-horizon discounted MDP model of a controlled dynamic and stochastic \emph{project} that consumes a single resource.
At the start of each time period $t = 0, 1, \ldots$, the controller observes the current project \emph{state} $s(t)$, which moves through the finite state space $\mathcal{N} \triangleq \{1, \ldots, N\}$, and then selects an \emph{action} $a(t)$ from the finite action space $\mathcal{A} = \{0, 1, \ldots, A\}$. The choice of action at each time $t$ is
 based on a possibly randomized function of the system's \emph{history} $\mathcal{H}(t) \triangleq 
 \{s(t)\} \cup \{(s(t'), a(t'))\colon t' = 0, \ldots, t-1\}$, consisting of the current state  $s(t)$ and previous states visited and actions taken, if any.
This corresponds to adopting a \emph{control policy} (\emph{policy} for short) $\pi$ from the class $\Pi$ of \emph{history-dependent randomized policies} (see (\cite{put94} Sec.\ 2.1.5)). 
We will call such policies \emph{admissible}.
 
We will refer to action $0$ as the \emph{passive} action, as it models the project's evolution in the absence of control, and to $1, \ldots, A$ as the \emph{active} actions. Such actions model distinct  \emph{gears} for operating the project which are naturally ordered by their increasing resource consumption.  
Henceforth, we will use the terms \emph{action} and \emph{gear} interchangeably.

When the project occupies state $s(t) = i$ at the start of a period and action $a(t) = a$ is selected, it incurs a holding cost $h_i^a$ and consumes a quantity $q_i^a$ of the resource in the period, time-discounted with factor $0 < \beta < 1$.
Then, the project state moves in a Markovian fashion from 
$s(t) = i$ to $s(t+1) = j$ with probability $p_{ij}^a$. 

Consistently with the interpretation of actions $a$ as operating gears ordered by increasing activity levels, we shall assume that higher gears entail larger resource consumptions, so the resource consumption $q_i^a$ is monotone increasing in the gear $a$ for each state $i$:
\begin{equation}
\label{eq:resconsorder}
0 \leqslant q_i^0 < q_i^1 < \cdots < q_i^A, \quad i \in \mathcal{N}.
\end{equation}

Intuitively, to compensate for their larger resource consumptions, higher gears should be more beneficial in some sense than lower gears, e.g., they might tend to drive the project towards less costly states or yield lower holding costs.

We further introduce a scalar parameter $\lambda \in \mathbb{R}$ modeling the resource unit price. Note that $\lambda$ could take negative values, in which case it would represent a subsidy for using the resource.
We shall consider the project's \emph{$\lambda$-price problem}, which is to find an admissible project operating policy $\pi^*(\lambda)$ minimizing the expected total discounted holding and resource usage cost for \emph{any} initial state.
Writing as $\Ex_i^\pi[\cdot]$ the expectation under policy $\pi$ starting from state $i$, we denote by 
\begin{equation}
\label{eq:vipil}
V_i(\lambda, \pi) \triangleq \Ex_i^\pi\bigg[\sum_{t=0}^\infty \big(h_{s(t)}^{a(t)} + \lambda q_{s(t)}^{a(t)}\big) \beta^t\bigg]
\end{equation}
the corresponding expected total discounted cost incurred by the project when resource usage is charged at price $\lambda$. 
The resulting \emph{optimal (project) cost function} is 
\[
V_i^*(\lambda) \triangleq \inf \, \{V_i(\lambda, \pi)\colon \pi \in \Pi\}, \quad i \in \mathcal{N}.
\]

We can thus formulate the project's $\lambda$-price problem as
\begin{equation}
\label{eq:lambdapp}
(P_\lambda) \quad \textup{find } \pi^*(\lambda) \in \Pi\colon V_i(\lambda, \pi^*(\lambda)) = 
V_i^*(\lambda), \quad i \in \mathcal{N}.
\end{equation}
We shall refer to a policy $\pi^*(\lambda)$ solving the $\lambda$-price problem $(P_\lambda)$ as a \emph{$\lambda$-optimal policy}.

Denoting by $\Pi^{\textup{SD}}$ the class of \emph{stationary deterministic policies} (see \cite{put94} (Sec.\ 2.1.5)), which base action choice on the current state only, standard results in MDP theory (see~\cite{put94} (Theorem 6.2.10.a)) ensure the existence of a $\lambda$-optimal policy $\pi^*(\lambda) \in \Pi^{\textup{SD}}$. 
Both the optimal cost function $V_i^*(\lambda)$ and the optimal stationary deterministic policies for the $\lambda$-price problem $(P_\lambda)$ are determined by Bellman's 
\emph{discounted-cost optimality equations}
\begin{equation}
\label{eq:bellmaneqns}
V_i^*(\lambda) = \min_{a \in \mathcal{A}} \, \bigg(h_i^a + \lambda q_i^a +  \beta \sum_{j \in \mathcal{N}} p_{ij}^a V_j^*(\lambda)\bigg), \enspace 
i \in \mathcal{N}.
\end{equation}

It is well known that the optimal cost function $V_i^*(\lambda)$ is the unique solution to such equations and that a  stationary deterministic policy is optimal iff (i.e., if and only if) it selects an action attaining the minimum in the right-hand side of (\ref{eq:bellmaneqns}) for each state $i$. We shall also call such actions $\lambda$-optimal.
For fixed $\lambda$, the Bellman equations can be solved numerically by classical methods, such as \emph{value iteration}, \emph{policy iteration}, and \emph{linear optimization}. See, e.g.,   (\cite{put94} Sec.\ 6).

\subsection{Indexability}
\label{s:indexab}
Instead of solving the $\lambda$-price problem for specific values of the parameter $\lambda$, 
we shall pursue an alternative approach aiming at a complete understanding of optimal policies over the entire parameter space, by fully characterizing the optimal actions for the parametric collection $\mathcal{P} \triangleq \{P_\lambda\colon \lambda \in \mathbb{R}\}$ of \emph{all} $\lambda$-price problems. Such a characterization will be given in terms of \emph{critical parameter values}  $\lambda^{*, a}_{i}$, as defined next.

Note that, below and throughout the paper, we use the standard abbreviation \textit{iff} for \textit{if and only if}.

\begin{definition}[Indexability and DAI]
\label{def:indxb} {\hl{We call the} 
 above multi-gear bandit model {indexable} if there exist critical resource prices $\lambda_i^{*, a}$ for every state $i$ and active action (gear) $a \geqslant 1$ satisfying $\lambda_i^{*, A} \leqslant  \cdots \leqslant \lambda_i^{*, 1}$,  such that, for any such state and resource price $\lambda \in \mathbb{R}$: (i) action $0$ is $\lambda$-optimal in state $i$ iff $\lambda \geqslant   \lambda_i^{*, 1}$;
(ii) action $1 \leqslant a \leqslant A-1$ is $\lambda$-optimal in state $i$ iff $\lambda_i^{*, a+1} \leqslant  \lambda \leqslant   \lambda_i^{*, a}$; and (iii) action $A$ is $\lambda$-optimal in state $i$ iff $\lambda \leqslant   \lambda_i^{*, A}$. We call $\lambda_i^{*, a}$ the model's {dynamic allocation index} (DAI), viewed as a function of $(i, a)$.}
\end{definition}

\begin{remark}
\label{rem:indxb} {

\item[(i)] {The definition}
 of indexability for multi-action bandits was first outlined by Weber~\cite{weber07} in the setting of a three-action project model and was first formalized by the author~\cite{nmvaluet08a} in the general setting considered herein. The latter paper further introduced the {multi-armed multi-mode bandit problem} and proposed to use the above DAI as the basis for a heuristic index policy for it. The concept of indexability for two-gear (active/passive) bandits has its roots in the work of Bellman~\cite{bellman56}, where he characterized the optimal policies for operating a Bernoulli bandit in terms of critical parameter values. Gittins and Jones~\cite{gijo74} showed that such critical values (later known as {Gittins indices}) provide a tractable optimal policy for the classic {multi-armed bandit problem}, involving the optimal sequential activation of a collection of two-gear bandits that do not change state when passive. The idea of indexability was extended by Whittle~\cite{whit88b} to two-gear bandits that can change state when passive, called {restless bandits}. He also proposed to use the corresponding {Whittle index policy} as a heuristic for the intractable {multi-armed restless bandit problem}, when the individual bandits (projects) are {indexable}, which need not be the case as there are nonindexable bandits.
\item[(ii)] Writing as $V_i(\lambda, \langle a, *\rangle) \triangleq h_i^a + \lambda q_i^a + \beta \sum_{j \in \mathcal{N}} p_{ij}^a V_j^*(\lambda)$ the optimal cost function with initial action $a$, indexability means that there exist critical prices $\lambda_i^{*, a}$ as in Definition \ref{def:indxb} such that, for each state $i $,
\begin{equation}
\label{eq:indxovf}
\begin{split}
V_i(\lambda, \langle 0, *\rangle) \leqslant V_i(\lambda, \langle a, *\rangle) \enspace \textup{for} \enspace a \geqslant 1 & \Longleftrightarrow \lambda \geqslant   \lambda_i^{*, 1} \\
\textup{for} \enspace 0 < a < A\colon V_i(\lambda, \langle a, *\rangle) \leqslant V_i(\lambda, \langle a', *\rangle) \enspace \textup{for} \enspace a' \neq a & \Longleftrightarrow \lambda_i^{*, a+1} \leqslant  \lambda \leqslant   \lambda_i^{*, a} \\
V_i(\lambda, \langle A, *\rangle) \leqslant V_i(\lambda, \langle a, *\rangle) \enspace \textup{for} \enspace a < A & \Longleftrightarrow \lambda \leqslant   \lambda_i^{*, A}.
\end{split}
\end{equation}
\item[(iii)] In intuitive terms, when an indexable project model occupies state $i$, it is $\lambda$-optimal to select the lowest (passive) gear $0$ iff the resource is expensive enough ($\lambda \geqslant   \lambda_i^{*, 1}$); it is $\lambda$-optimal to select the highest gear $A$ iff the resource is cheap enough ($\lambda \leqslant   \lambda_i^{*, a}$); and it is $\lambda$-optimal to select the intermediate gear $0 < a < A$ iff the resource price lies between the critical prices $\lambda_i^{*, a+1}$ and $\lambda_i^{*, a}$.
\item[(iv)] In an indexable model, $\lambda_i^{*, a}$ is the unique critical resource price $\lambda$ for which gears $a-1$ and $a$ are both $\lambda$-optimal in state $i$; hence, it is the unique solution to the equation
\begin{equation}
\label{eq:usteiv}
V_i(\lambda, \langle a-1, *\rangle) = V_i(\lambda, \langle a, *\rangle).
\end{equation}
Yet, note that for a nonindexable model, Equation (\ref{eq:usteiv}) need not have a unique solution.
\item[(v)] Let $\mathcal{A}_i^*(\lambda)$ be the set of $\lambda$-optimal actions in state $i$. If the model is indexable then, for each active action $a \geqslant 1$, $\mathcal{A}_i^*(\lambda) \cap \{a, \ldots, A\} \neq \emptyset$ (i.e., there is an optimal action greater than or equal to $a$) iff $\lambda_i^{*, a} \geqslant \lambda$ and $\mathcal{A}_i^*(\lambda) \cap \{0, \ldots, a-1\} \neq \emptyset$ (i.e., there is an optimal action less than $a$) iff $\lambda_i^{*, a} \leqslant \lambda.$
}
\end{remark}

We shall address the following research goals: (1) identify sufficient conditions ensuring that the above multi-gear bandit model is indexable and (2) for models satisfying such conditions, provide an efficient means of computing the DAI.

\subsection{Project Performance Metrics and Their Characterization}
\label{s:pmetr}
To formulate the main result and facilitate the required analyses, we consider certain project \emph{performance metrics}. 
We measure the holding cost incurred by the  project under a policy 
$\pi \in \Pi$ starting from the initial-state distribution $s(0) \sim p = (p_i)_{i \in \mathcal{N}}$, so $\Prob\{s(0) = i\} = p_i$ for $i \in \mathcal{N}$, by the  \emph{(holding) cost metric}
\[
F_p(\pi) = \Ex_p^\pi\bigg[\sum_{t=0}^\infty h_{s(t)}^{a(t)} \beta^t\bigg],
\]
where $\Ex_p^\pi[\cdot]$ denotes expectation under policy $\pi$ starting from $s(0) \sim p$. 
Similarly, we measure the corresponding resource usage by the \emph{resource (usage) metric}
\[
G_p(\pi) = \Ex_p^\pi\bigg[\sum_{t=0}^\infty q_{s(t)}^{a(t)} \beta^t \bigg].
\]

When $s(0) = i$, we write $F_i(\pi)$ and $G_i(\pi)$. Note that $F_p(\pi) = \sum_i p_i F_i(\pi)$ and $G_p(\pi) = \sum_i p_i G_i(\pi)$.
As for the total cost metric  $V_i(\lambda, \pi)$ in (\ref{eq:vipil}), we can  
thus express it as
\[
V_i(\lambda, \pi) = F_i(\pi) + \lambda G_i(\pi).
\]

We shall similarly write $V_p(\lambda, \pi)$ when $s(0) \sim p$.

We next address the characterization of such metrics for stationary deterministic policies. 
We will represent any such policy by the partition $S = (S_a)_{a \in \mathcal{A}} = (S_0, \ldots, S_A)$ it naturally induces on the state space $\mathcal{N}$, where $S_a$ is the subset of states where the policy selects gear $a$.
We shall refer to it as the \emph{$S$-policy} or \emph{policy $S$}. 

The performance metrics $F_i(S)$ and $G_i(S)$ for the $S$-policy are thus characterized as the unique solutions to the linear equation systems
\[
F_i(S) = h_i^a + \beta \sum_{j \in \mathcal{N}} p_{ij}^a F_j(S), \enspace i \in S_a, a \in \mathcal{A},
\]
and
\[
G_i(S) = q_i^a + \beta \sum_{j \in \mathcal{N}} p_{ij}^a G_j(S), \enspace i \in S_a, a \in \mathcal{A}.
\]

In the sequel we will find it convenient to use vector notation,  denoting vectors and matrices in boldface and writing, \hl{e.g.,} 
 $\mathbf{F}(S) = (F_i(S))_{i \in \mathcal{N}}$ and $\mathbf{F}_B(S) = (F_i(S))_{i \in B}$ for \mbox{$B \subset \mathcal{N}$}, and similarly for $\mathbf{G}(S)$, $\mathbf{h}^a$ and $\boldsymbol{q}^a$. We will also write $\mathbf{P}^a = (p_{ij}^a)_{i, j \in \mathcal{N}}$, \mbox{$\mathbf{P}_{B B'}^a = (p_{ij}^a)_{i \in B, j \in B'}$} and $\mathbf{P}_{B \boldsymbol{\cdot}}^a = (p_{ij}^a)_{i \in B, j \in \mathcal{N}}$ for 
$B, B' \subset \mathcal{N}$.

The above equations characterizing $\mathbf{F}(S)$ and $\mathbf{G}(S)$ are thus formulated as
\begin{equation}
\label{eq:FSaSeq}
\mathbf{F}_{S_a}(S) = \mathbf{h}_{S_a}^a + \beta \mathbf{P}_{S_a \boldsymbol{\cdot}}^a \mathbf{F}(S), \enspace a \in \mathcal{A}.
\end{equation}
and
\begin{equation}
\label{eq:GSaSeq}
\mathbf{G}_{S_a}(S) = \boldsymbol{q}_{S_a}^a + \beta \mathbf{P}_{S_a \boldsymbol{\cdot}}^a \mathbf{G}(S), \enspace a \in \mathcal{A}.
\end{equation}

We shall further consider corresponding marginal metrics.
Denote by $\langle a, S\rangle$ the policy that selects gear $a$ at time $t = 0$ and then adopts the $S$-policy thereafter. Note that
\[
F_i(\langle a, S\rangle) = h_i^{a} + \beta \sum_{j \in \mathcal{N}} p_{ij}^{a} F_j(S)
\]
and
\[
G_i(\langle a, S\rangle) = q_i^{a} + \beta \sum_{j \in \mathcal{N}} p_{ij}^{a} G_j(S).
\]

For given actions $a \neq a'$,  we define the \emph{marginal (holding) cost metric}
\begin{equation}
\label{eq:fiaapS}
f_i^{a, a'}(S) \triangleq F_i(\langle a, S\rangle) - F_i(\langle a', S\rangle) = 
h_i^{a} - h_i^{a'} + 
\beta \sum_{j \in \mathcal{N}} p_{ij}^{a} F_j(S) - \beta \sum_{j \in \mathcal{N}} p_{ij}^{a'} F_j(S),
\end{equation}
which measures the \emph{decrement} in the  holding cost metric that results from the shifting of the initial gear from $a$ to $a'$ starting from state $i$, provided that the $S$-policy is followed thereafter.

We also define the \emph{marginal resource (usage) metric}
\begin{equation}
\label{eq:giaapS}
g_i^{a, a'}(S) \triangleq G_i(\langle a', S\rangle) - G_i(\langle a, S\rangle) =  q_i^{a'} - q_i^{a} + 
\beta \sum_{j \in \mathcal{N}} p_{ij}^{a'} G_j(S) - \beta \sum_{j \in \mathcal{N}} p_{ij}^a G_j(S),
\end{equation}
which measures the corresponding \emph{increment} in the resource metric.

In vector notation, we can write the above identities as
\begin{equation}
\label{eq:vfiaapS}
\mathbf{f}^{a, a'}(S) \triangleq \mathbf{F}(\langle a, S\rangle) - \mathbf{F}(\langle a', S\rangle) = 
\mathbf{h}^{a} - \mathbf{h}^{a'} + 
\beta (\mathbf{P}^{a}  - \mathbf{P}^{a'}) \mathbf{F}(S)
\end{equation}
and
\begin{equation}
\label{eq:vgiaapS}
\mathbf{g}^{a, a'}(S) \triangleq \mathbf{G}(\langle a', S\rangle) - \mathbf{G}(\langle a, S\rangle) =  \boldsymbol{q}^{a'} - \boldsymbol{q}^{a} + 
\beta  (\mathbf{P}^{a'}  - \mathbf{P}^a) \mathbf{G}(S).
\end{equation}

If $g_i^{a, a'}(S) > 0$ for certain $i$, $a$, $a'$, and $S$, we further define the \emph{marginal productivity (MP) metric} as the ratio of the marginal cost metric to  the marginal resource metric:
\begin{equation}
\label{eq:lambdiaapS}
m_i^{a, a'}(S) \triangleq \frac{f_i^{a, a'}(S)}{g_i^{a, a'}(S)}.
\end{equation}

We next present a preliminary result, on which we draw later on, establishing further relations between metrics $\mathbf{F}(S)$ and $\mathbf{f}^{a, a'}(S)$ and between $\mathbf{G}(S)$ and $\mathbf{g}^{a, a'}(S)$. Note that, below, $\mathbf{0}_{S_a}$ denotes  a vector of zeros with components indexed by $S_a$.

\begin{lemma}
\label{lma:civisrel}
For any stationary deterministic policy $S$ and action $a,$
\begin{itemize}
\item[(a)] $
 (\mathbf{I} - \beta \mathbf{P}^a) \mathbf{F}(S)  - \mathbf{h}^a = 
\begin{bmatrix}
(\mathbf{f}_{S_{a'}}^{a', a}(S))_{a' \in \mathcal{A}-\{a\}} \\
\mathbf{0}_{S_a}
\end{bmatrix};
$
\item[(b)] $
\boldsymbol{q}^a - (\mathbf{I} - \beta \mathbf{P}^a) \mathbf{G}(S)  = 
\begin{bmatrix}
(\mathbf{g}_{S_{a'}}^{a', a}(S))_{a' \in \mathcal{A}-\{a\}} \\
\mathbf{0}_{S_a}
\end{bmatrix}.
$
\end{itemize}
\end{lemma}
\begin{proof}
(a) For $a' \neq a$, using in turn (\ref{eq:vfiaapS}) and (\ref{eq:FSaSeq}) we obtain
\begin{align*}
\mathbf{f}_{S_{a'}}^{a', a}(S) & = \mathbf{h}_{S_{a'}}^{a'} - \mathbf{h}_{S_{a'}}^{a} + 
\beta (\mathbf{P}_{S_{a'} \boldsymbol{\cdot}}^{a'} - \mathbf{P}_{S_{a'} \boldsymbol{\cdot}}^{a})\mathbf{F}(S) \\
& = (\mathbf{F}_{S_{a'}}(S) - \mathbf{h}_{S_{a'}}^{a'} - \beta \mathbf{P}_{S_{a'} \boldsymbol{\cdot}}^{a'} \mathbf{F}(S))  + 
(\mathbf{h}_{S_{a'}}^{a'} - \mathbf{h}_{S_{a'}}^{a} + 
\beta (\mathbf{P}_{S_{a'} \boldsymbol{\cdot}}^{a'} - \mathbf{P}_{S_{a'} \boldsymbol{\cdot}}^{a})\mathbf{F}(S)) \\
& = 
\mathbf{F}_{S_{a'}}(S)  -  \beta \mathbf{P}_{S_{a'} \boldsymbol{\cdot}}^{a} \mathbf{F}(S) - \mathbf{h}_{S_{a'}}^{a},
\end{align*}
and
\[
\mathbf{F}_{S_{a}}(S)  -  \beta \mathbf{P}_{S_{a} \boldsymbol{\cdot}}^{a} \mathbf{F}(S) - \mathbf{h}_{S_{a}}^{a} = \mathbf{0}_{S_a}.
\]

Part (b) follows similarly.
\end{proof}

\subsection{Main Result: A Verification Theorem for Indexability}
\label{s:mrvti}
We next present our main result, giving sufficient conditions for indexability. The conditions
correspond to the framework of \emph{PCL-indexability}, which is extended here from the two-gear setting in~\cite{nmaap01,nmmp02,nmmor06,nmmor20} to the 
multi-gear setting.

We will use the following notation. Given a stationary deterministic policy {$S = (S_0, \ldots, S_A)$},  
actions $a \neq a'$,  and a state $j \in S_{a}$, the policy denoted by $\widehat{S} = \mathcal{T}_j^{a, a'} S$ is defined by 
$\widehat{S}_a = S_a - \{j\}$, $\widehat{S}_{a'} = S_{a'} \cup \{j\}$, and $\widehat{S}_{a''} = S_{a''}$ for $a \neq a'' \neq a'$. Thus, $\mathcal{T}_j^{a, a'} S$ is obtained from $S$ by shifting the gear selected in state $j$ from $a$ to $a'$.

The verification theorem below refers to indexability relative to a structured family $\mathcal{F}$ of stationary deterministic policies, which one needs to postulate a priori, based on insight on the particular model at hand.
We shall thus refer to the family of \emph{$\mathcal{F}$-policies} $S \in \mathcal{F}$, and to \emph{$\mathcal{F}$-indexability}, as defined below.

\begin{definition}[$\mathcal{F}$-indexability]
\label{def:Findxb} {
We call the model {$\mathcal{F}$-indexable} if (i) it is indexable and (ii) $\mathcal{F}$-policies are optimal for the $\lambda$-price problem $(P_\lambda)$ in (\ref{eq:lambdapp}), for any $\lambda \in \mathbb{R}.$
}
\end{definition}

Note that Definition \ref{def:Findxb}(ii) refers to the \emph{optimality of $\mathcal{F}$-policies} for all $\lambda$-price problems $(P_\lambda)$. By this we mean that, for any $\lambda \in \mathbb{R}$, there exists a $\lambda$-optimal policy $S^*(\lambda) \in \mathcal{F}$.

We require $\mathcal{F}$ to satisfy the following \emph{connectedness assumption}, which is motivated by algorithmic considerations. The assumption ensures that it is possible to go from policy $(\emptyset, \ldots, \emptyset, \mathcal{N})$ to $(\mathcal{N}, \emptyset, \ldots, \emptyset)$, both of which must be in $\mathcal{F}$,  through a sequence of policies in $\mathcal{F}$ where each policy in the sequence is obtained from the previous one by \emph{downshifting} the gear selected in a single state to the next lower gear. Conversely, it is possible to go from $(\mathcal{N}, \emptyset, \ldots, \emptyset)$ to $(\emptyset, \ldots, \emptyset, \mathcal{N})$  through a sequence of policies in $\mathcal{F}$ where each policy in the sequence is obtained from the previous one by \emph{upshifting} the gear selected in a single state to the next higher gear.

\begin{assumption} {
\label{s:assF} The family of policies $\mathcal{F}$ satisfies the following conditions:
\begin{itemize}
\item[(i)] $(\emptyset, \ldots, \emptyset, \mathcal{N}) \in \mathcal{F}$ and $(\mathcal{N}, \emptyset, \ldots, \emptyset) \in \mathcal{F};$
\item[(ii)] For each $S \in \mathcal{F} -  \{(\mathcal{N}, \emptyset, \ldots, \emptyset)\}$ there exist $a \geqslant 1$ and $j \in S_a$ such that $\mathcal{T}_j^{a, a-1} S \in \mathcal{F};$
\item[(iii)] For each $S \in \mathcal{F} - \{(\emptyset, \ldots, \emptyset, \mathcal{N})\}$
 there exist $a < A$ and $j \in S_a$ such that $\mathcal{T}_j^{a, a+1} S \in \mathcal{F}.$
\end{itemize}}
\end{assumption}

Note that the above concepts of downshifting and upshifting gears naturally induce a \emph{partial ordering} $\preceq$ on the class of all stationary deterministic policies $S$ and in particular on $\mathcal{F}$.
Thus, given $S$ and $S'$, we write $S \preceq S'$ if, at every state, $S$ does not select a higher gear  than $S'$. If, further, $S \neq S'$, we write  
$S \prec S'$.
Assumption \ref{s:assF} shows that the \emph{poset} (partially ordered set) $(\mathcal{F}, \preceq)$ contains the least element $(\mathcal{N}, \emptyset, \ldots, \emptyset)$ and the largest element $(\emptyset, \ldots, \emptyset, \mathcal{N})$.

The verification theorem refers to the \emph{downshift adaptive-greedy index algorithm} $\mathrm{DS}(\mathcal{F})$ shown in Algorithm \ref{alg:tdabutd}.
This takes as input the model parameters and, in $K \triangleq A N$ steps,  produces as output a sequence of distinct state--action pairs $(j_k, a_k)$ spanning $\mathcal{N} \times (\mathcal{A} - \{0\})$ along with corresponding sequences of $\mathcal{F}$-policies $S^k$ and scalars $m_{j_k}^{*, a_k}$ for $k = 1,\ldots,  K$. 
Actually, the sequence $S^k$ is a \emph{chain} of the poset $(\mathcal{F}, \preceq)$ ordered as 
\[
(\mathcal{N}, \emptyset, \ldots, \emptyset) = S^{K+1} \prec S^{K} \prec \cdots \prec S^{1} = (\emptyset, \ldots, \emptyset, \mathcal{N}).
\]

\begin{algorithm}[H]
\caption{\hl{Downshift} 
 adaptive-greedy index algorithm $\mathrm{DS}(\mathcal{F})$.}
{%
\begin{minipage}{2.5in}
\textbf{Output:}
$\{(j_k, a_k), S^k, m_{j_k}^{*, a_k}\}_{k=1}^{K}$
\begin{tabbing}
\textit{Initialization:} 
$S^1 := (\emptyset, \ldots, \emptyset, \mathcal{N});$ \enspace  $a_1 := A$  \\
\textbf{pick} $ j_1 \in \argmin_{j \in \mathcal{N}, \mathcal{T}_j^{a_1, a_1-1} S^1 \in \mathcal{F}} \, m_j^{a_1-1, a_1}(S^1)$ \\
 $m_{j_1}^{*, a_1} := m_{j_1}^{a_1-1, a_1}(S^1);$ \enspace 
$S^2 := \mathcal{T}_{j_1}^{a_1, a_1-1} S^1$ \\
\textit{Loop:} \\
\textbf{for} \= $k := 2$ \textbf{to}  $K$ \textbf{do} \\
\> \textbf{pick}  
 $(j_k, a_k) \in \argmin_{(j, a)\colon  j \in S_a^{k}, \mathcal{T}_j^{a, a-1} S^{k} \in \mathcal{F}} \, 
      m_j^{a-1, a}(S^{k})$, \\
\>    \quad  with $m_j^{a-1, a}(S^{k})
 := m_{j_{k-1}}^{*, a_{k-1}} + 
 \frac{g_j^{a-1, a}(S^{k-1})}{g_j^{a-1, a}(S^k)} (m_j^{a-1, a}(S^{k-1})  - m_{j_{k-1}}^{*, a_{k-1}})$ \\
\>   $m_{j_k}^{*, a_k} := 
 m_{j_k}^{a_k-1, a_k}(S^{k})$;  \, $S^{k+1} := \mathcal{T}_{j_k}^{a_k, a_k-1} S^{k}$ \\
\textbf{end} \{ for \}
\end{tabbing}
\end{minipage}}
\label{alg:tdabutd}
\end{algorithm}

\begin{remark}
\label{rem:tdabutd} {

\item[(i)] {Algorithm} 
 $\mathrm{DS}(\mathcal{F})$ extends to multi-gear bandits the adaptive-greedy algorithm for computing the Whittle index for restless (two-gear) bandits introduced by the author in~\cite{nmaap01} and further developed in~\cite{nmmp02}. In turn, this has its early roots in Klimov's adaptive-greedy index algorithm~\cite{kl74}  for computing the indices that give the optimal policy for scheduling a multi-class queue with feedback. Note that Klimov's algorithm was first adapted in~\cite{bnm96} to compute the Gittins index for  classic bandits (i.e., two-gear bandits that do not change state when passive).
\item[(ii)] 
The term {downshift} refers to the way in which the algorithm generates the sequence $S^k$ of $\mathcal{F}$-policies. It starts with policy $S^1$, which selects the highest gear $A$ in every state. 
Then, at each step $k$ of the algorithm, it selects a state $j_k$ in which to downshift gears from $a_k$ to $a_{k}-1$, keeping the same gears in the other states, thus obtaining the next policy $S^{k+1}$. The selection of such a state $j_k$ is performed in an {adaptive-greedy} fashion, by choosing a state in which such a downshifting change entails a minimal MP decrease, as measured by the MP metric $m_j^{a-1, a}(S^{k})$, while also ensuring that the next policy $S^{k+1}$ will be in $\mathcal{F}$.
One can visualize the workings of the algorithm in terms of balls trickling down a grid in which states are positioned in columns and gears in rows, with gear $A$ at the top.
Initially, all balls are in the top row, as the algorithm starts with policy $S^1 \triangleq (\emptyset, \ldots, \emptyset, \mathcal{N})$, so gear $A$ is chosen in every state.
Then, at each step of the algorithm, one ball trickles down from its current row to that immediately below, which represents another $\mathcal{F}$-policy. The algorithm ends in $K$ steps when all $N$ balls have trickled down to the bottom row, which corresponds to policy $S^{K+1} \triangleq (\mathcal{N}, \emptyset, \ldots, \emptyset)$, so gear $0$ is chosen in every state.
\item[(iii)] 
Note that, by construction, (1) each policy $S^k$ in the sequence produced by the algorithm satisfies $S^k \in \mathcal{F}$, i.e., it is an $\mathcal{F}$-policy, for $k = 1, \ldots, K+1$, and (2) the state--action pairs $(j_k, a_k)$ produced by the algorithm are all distinct, spanning the $K$ state--action pairs $(j, a) \in \mathcal{N} \times (\mathcal{A} - \{0\})$ corresponding to active gears $a \geqslant 1$.
}
\end{remark}

\begin{definition}[PCL$(\mathcal{F})$-indexability and MP index]
\label{def:pclind} { 
We call a multi-gear bandit model {PCL-indexable} with respect to $\mathcal{F}$-policies, or  {PCL$(\mathcal{F})$-indexable}, if the following hold: 
\begin{enumerate}[label={\rm (\subscript{PCLI}{{\arabic*}})}]
\item  $g_{j}^{a-1, a}(S) > 0$ for every policy $S \in \mathcal{F}$, active action $a \geqslant 1$, and state $j \in \mathcal{N};$
\item  Algorithm $\mathrm{DS}(\mathcal{F})$ computes the $m_{j_k}^{*,  a_k}$ in nondecreasing order: \[
m_{j_{1}}^{*, a_{1}} \leqslant m_{j_2}^{*,  a_2} \leqslant \cdots \leqslant m_{j_K}^{*,  a_K}.
\]
 \end{enumerate}
 
In such a case, we call $m_j^{*, a}$ the project's {MP index} or {MPI} for short.}
\end{definition}

\begin{remark}
\label{rem:tdabutd2} {

\item[(i)] {Condition} 
 (PCLI1) means that the marginal resource metric corresponding to upshifting gears in any state, relative to any $\mathcal{F}$-policy, is positive. Note that it is equivalent to requiring that $g_{j}^{a', a}(S) > 0$ for $a' < a$, since $g_{j}^{a', a}(S) = g_{j}^{a', a'+1}(S) + \cdots + g_{j}^{a-1, a}(S)$.
 Such a condition and the fact that the $S^k$ are in $\mathcal{F}$ ensures that the MP index $m_{j}^{*,  a}$ computed by the algorithm is well defined. 
\item[(ii)] The recursive  formula used in the algorithm for computing the MP metrics as 
\begin{equation}
\label{eq:recindcpt}
m_j^{a-1, a}(S^{k})
 := m_{j_{k-1}}^{*, a_{k-1}} + 
 \frac{g_j^{a-1, a}(S^{k-1})}{g_j^{a-1, a}(S^k)} (m_j^{a-1, a}(S^{k-1})  - m_{j_{k-1}}^{*, a_{k-1}})
 \end{equation}
 is justified by Lemma \ref{lma:recindcomp}.
}
\end{remark}

We next state the verification theorem.
\begin{theorem}
\label{the:verthe}
If a multi-gear bandit model is PCL$(\mathcal{F})$-indexable, then it is  $\mathcal{F}$-indexable with its DAI being given by its MPI, i.e.,   $\lambda_j^{*, a} = m_j^{*, a}$.
\end{theorem}

The proof of Theorem \ref{the:verthe} is presented in Section \ref{s:proofotvthe}. 
It requires substantial preliminary groundwork, which is laid out in Sections \ref{s:lpreflpp}--\ref{s:pcl}.

\section{Linear Optimization Reformulation of the \boldmath{$\lambda$}-Price Problem}
\label{s:lpreflpp}
We start the required groundwork by reviewing the standard \emph{linear optimization} (LO) formulation of a finite-state and -action MDP (see, e.g., (\cite{put94} Sec.\ 6.9)), since it applies to the $\lambda$-price problem $(P_\lambda)$ in (\ref{eq:lambdapp}), as it is needed in subsequent analyses. 
It is well known  that such an MDP  can be reformulated as an LO problem on variables $x_j^a$ for state--action pairs $(j, a) \in \mathcal{K} \triangleq \mathcal{N} \times \mathcal{A}$, which represent \emph{discounted state--action occupancy measures}. Thus, variable $x_j^a$ corresponds to 
the measure
\[
x_{pj}^{a}(\pi) \triangleq \Ex_{p}^{\pi}\bigg[\sum_{t=0}^\infty \mathbf{1}_{\{(j, a)\}}(s(t), a(t)) \beta^t\bigg],
\] 
where $\mathbf{1}_B(\cdot)$ denotes the indicator function of a set $B$, so 
$x_{pj}^{a}(\pi)$ is the expected total discounted number of times that action $a$ is selected in state $j$ under policy $\pi$ starting from $s(0) \sim p$. We write it as $x_{ij}^{a}(\pi)$ when $s(0) = i$.

Such a standard LO formulation is 
\begin{equation}
\label{eq:plp}
\begin{split}
(L_{\lambda}(p))\colon \qquad & \minim \, \sum_{(j, a) \in \mathcal{K}} (h_j^a + \lambda q_j^a) x_j^a \\
& \st\colon  x_j^a \geqslant 0 \\
& \sum_{a \in \mathcal{A}} x_j^a - \beta \sum_{(i, a) \in \mathcal{K}}  p_{ij}^a x_i^a = p_j, \enspace 
j \in \mathcal{N},
\end{split}
\end{equation}
or, in vector notation, writing the probability mass function $p$ as a vector $\mathbf{p}$,
\begin{equation}
\label{eq:plpmf}
\begin{split}
(L_{\lambda}(p))\colon \qquad & \minim \, \sum_{a \in \mathcal{A}} \mathbf{x}^a (\mathbf{h}^a + \lambda \boldsymbol{q}^a) \\
& \st\colon  \mathbf{x}^a \geqslant \mathbf{0}, \enspace a \in \mathcal{A} \\
& \sum_{a \in \mathcal{A}} \mathbf{x}^a (\mathbf{I} - \beta \mathbf{P}^a) = \mathbf{p}.
\end{split}
\end{equation}

We next use the above LO formulation to show that the $\lambda$-price problem $(P_\lambda)$ in (\ref{eq:lambdapp}) can be reformulated into an equivalent problem in which  holding costs under action $A$ are zero, a result on which we draw later on.

Thus, define 
\begin{equation}
\label{eq:hatha}
\widehat{\mathbf{h}}^a \triangleq \mathbf{h}^a - (\mathbf{I} - \beta \mathbf{P}^a) (\mathbf{I} - \beta \mathbf{P}^A)^{-1}  \mathbf{h}^A, \quad a \in \mathcal{A},
\end{equation}
and note that 
\begin{equation}
\label{eq:haththetaazero}
\widehat{\mathbf{h}}^A = \mathbf{0}.
\end{equation}

Denote by $\widehat{F}_i(\pi)$, $\widehat{f}_i^{a, a'}(S)$, and $\widehat{m}_i^{a, a'}(S)$ the cost, marginal cost, and marginal productivity metrics defined in Section \ref{s:pmetr}, but for the model with modified holding costs $\widehat{h}_i^a$  in (\ref{eq:hatha}).
The following result clarifies the relations between such metrics and those for the original holding costs $h_i^a$.
Recall that we denote by $S^1 \triangleq (\emptyset, \ldots, \emptyset, \mathcal{N})$ the policy that selects the highest gear in every state.

\begin{lemma}
\label{lma:relhhhat} For any admissible policy $\pi$, stationary deterministic policy $S$, and actions $a \neq a'$:
\begin{itemize}
\item[(a)] $\widehat{F}_p(\pi) = F_p(\pi) - F_p(S^1);$
\item[(b)] $\widehat{f}_i^{a, a'}(S) = f_i^{a, a'}(S);$
\item[(c)] $\widehat{m}_i^{a, a'}(S) = m_i^{a, a'}(S).$
\end{itemize}
\end{lemma}
\begin{proof}
(a) The measures $\mathbf{x}_p^{a}(\pi) = (x_{pj}^{a}(\pi))_{j \in \mathcal{N}}$, viewed as row vectors,  satisfy the constraints in (\ref{eq:plpmf}):
\begin{equation}
\label{eq:LPconstrx}
\sum_{a =0}^A \mathbf{x}_p^a(\pi) (\mathbf{I} - \beta \mathbf{P}^a) = \mathbf{p}.
\end{equation}

Hence, since the matrices $\mathbf{I} - \beta \mathbf{P}^a$ are invertible, we can write
\[
\mathbf{x}_p^A(\pi) + \sum_{a =0}^{A-1} \mathbf{x}_p^a(\pi) (\mathbf{I} - \beta \mathbf{P}^a) (\mathbf{I} - \beta \mathbf{P}^A)^{-1} = \mathbf{p} (\mathbf{I} - \beta \mathbf{P}^A)^{-1}.
\]

We thus obtain
\begin{align*}
F_p(\pi) & = 
\sum_{a =0}^A \mathbf{x}_p^a(\pi) \mathbf{h}^a   = 
\sum_{a =0}^{A-1} \mathbf{x}_p^a(\pi) \mathbf{h}^a + \mathbf{x}_p^A(\pi) \mathbf{h}^A   \\ &
= \mathbf{p} (\mathbf{I} - \beta \mathbf{P}^A)^{-1} \mathbf{h}^A    +  \sum_{a =0}^{A-1} \mathbf{x}_p^a(\pi) \big[\mathbf{h}^a  - (\mathbf{I} - \beta \mathbf{P}^a) (\mathbf{I} - \beta \mathbf{P}^A)^{-1}  \mathbf{h}^A \big] \\
& = \mathbf{p} (\mathbf{I} - \beta \mathbf{P}^A)^{-1} \mathbf{h}^A  + \sum_{a =0}^{A-1} \mathbf{x}_p^a(\pi) \widehat{\mathbf{h}}^a    \\
& = F_p(S^1) + \widehat{F}_p(\pi),
\end{align*}
where the last line is obtained from the previous one by using Equation~(\ref{eq:FSaSeq}), which yields  
\begin{equation}
\label{eq:FS1hA}
\mathbf{F}(S^1) =  (\mathbf{I} - \beta \mathbf{P}^A)^{-1} \mathbf{h}^A.
\end{equation}

(b) We have, using (\ref{eq:vfiaapS}), (\ref{eq:hatha}), and part (a), 
\begin{align*}
\widehat{\mathbf{f}}^{a, a'}(S) & \triangleq  \widehat{\mathbf{h}}^{a} - \widehat{\mathbf{h}}^{a'} + 
\beta (\mathbf{P}^{a} - \mathbf{P}^{a'}) \widehat{\mathbf{F}}(S) \\
& = \mathbf{h}^a - \mathbf{h}^{a'} + \beta (\mathbf{P}^a - \mathbf{P}^{a'}) (\mathbf{I} - \beta \mathbf{P}^A)^{-1}  \mathbf{h}^A + 
\beta (\mathbf{P}^{a} - \mathbf{P}^{a'}) (\mathbf{F}(S) -  \mathbf{F}(S^1))  \\
& = \mathbf{h}^a - \mathbf{h}^{a'}  + 
\beta (\mathbf{P}^{a} - \mathbf{P}^{a'}) \mathbf{F}(S)  \\
& = \mathbf{f}^{a, a'}(S).
\end{align*}

(c) This part follows directly from part (b) and (\ref{eq:lambdiaapS}), since 
$\widehat{m}_i^{a, a'}(S)  \triangleq \widehat{f}_i^{a, a'}(S)/g_i^{a, a'}(S) = 
f_i^{a, a'}(S)/g_i^{a, a'}(S) = m_i^{a, a'}(S).$
\end{proof}

\begin{corollary}
\label{cor:fiam1as1hiam1} For any state $j,$
\begin{itemize}
\item[(a)] $f_j^{a-1, a}(S^1) = \widehat{h}_j^{a-1} - \widehat{h}_j^{a}, \enspace a \geqslant 1;$
\item[(b)] $m_j^{a, A}(S^1) = \widehat{h}_j^{a} / g_j^{a, A}(S^1), \enspace a < A.$
\end{itemize}
\end{corollary}
\begin{proof} (a) We have
\[
\mathbf{f}^{a-1, a}(S^1) = \widehat{\mathbf{f}}^{a-1, a}(S^1) \triangleq \widehat{\mathbf{h}}^{a-1} - \widehat{\mathbf{h}}^{a}  + 
\beta (\mathbf{P}^{a-1} - \mathbf{P}^{a}) \widehat{\mathbf{F}}(S^1) = \widehat{\mathbf{h}}^{a-1} - \widehat{\mathbf{h}}^{a},
\]
where we have used in turn Lemma \ref{lma:relhhhat} (b), (\ref{eq:vfiaapS}), (\ref{eq:haththetaazero}), and (\ref{eq:FS1hA}). 

(b) This part follows from (a), since $m_j^{a, A}(S^1) = f_{j}^{a, A}(S^1)/g_{j}^{a, A}(S^1) = \widehat{h}_{j}^{a}/g_{j}^{a, A}(S^1)$.
\end{proof}

Denote by $(\widehat{P}_\lambda)$ the modified $\lambda$-price problem where the $\mathbf{h}^a$ in problem $(P_\lambda)$ are replaced by $\widehat{\mathbf{h}}^a$. Write as $\widehat{V}_p(\lambda, \pi)$ the project cost metric for the modified model.

\begin{lemma}
\label{lma:equivPlhatPl}
Problems $(P_\lambda)$ and $(\widehat{P}_\lambda)$ are equivalent, since 
$\widehat{V}_p(\lambda, \pi) = V_p(\lambda, \pi) - F_p(S^1).$
\end{lemma}
\begin{proof}
We have, using Lemma \ref{lma:relhhhat} (a),
\[
\widehat{V}_p(\lambda, \pi) = \widehat{F}_p(\pi) + \lambda G_p(\pi) = V_p(\lambda, \pi) - F_p(S^1).
\]
\end{proof}

\section{Relations between Performance Metrics}
\label{s:dperfmetr}
This section presents relations between performance metrics that we will need to prove the verification theorem.
We start with a result giving decomposition identities that relate the metrics 
$F_p(\pi)$ and $G_p(\pi)$ for an admissible policy $\pi$ to the metrics $F_p(S)$ and $G_p(S)$ under a particular stationary deterministic policy $S$,
where $p$ is the initial-state distribution.
The resulting \emph{decomposition identities} involve marginal cost metrics $f_{j}^{a, a'}(S)$ and marginal resource metrics $g_{j}^{a', a}(S)$.

\begin{lemma}[Performance metrics decomposition]
\label{lma:declaws}
For any admissible policy $\pi$ and stationary deterministic policy $S$:
\begin{itemize}
\item[(a)] $
F_p(S)  + \sum_{a < a'} \sum_{j \in S_{a'}} f_{j}^{a, a'}(S) x_{pj}^{a}(\pi)  = 
F_p(\pi) + \sum_{a' < a} \sum_{j \in S_{a'}} f_{j}^{a', a}(S) x_{pj}^{a}(\pi);
$
\item[(b)] $
G_p(\pi)  + \sum_{a < a'} \sum_{j \in S_{a'}} g_{j}^{a, a'}(S) x_{pj}^a(\pi)  = 
G_p(S) + \sum_{a' < a} \sum_{j \in S_{a'}} g_{j}^{a', a}(S) x_{pj}^a(\pi).
$
\end{itemize}
\end{lemma}
\begin{proof}
(a) 
We can write, using in turn (\ref{eq:LPconstrx}), Lemma \ref{lma:civisrel} (a), and $f_j^{a', a}(S) = - f_j^{a, a'}(S)$,
\begin{align*}
0 & = \bigg[\sum_{a} \mathbf{x}_p^a(\pi) (\mathbf{I} - \beta \mathbf{P}^a) - \mathbf{p}\bigg] \mathbf{F}(S)  = 
\sum_{a} \mathbf{x}_p^a(\pi) (\mathbf{I} - \beta \mathbf{P}^a) \mathbf{F}(S) - \mathbf{p}\mathbf{F}(S) \\
& =  
\sum_{a} \mathbf{x}_p^a(\pi) \big[(\mathbf{I} - \beta \mathbf{P}^a) \mathbf{F}(S) - \mathbf{h}^a\big] - \mathbf{p} \mathbf{F}(S) + 
\sum_{a} \mathbf{x}_p^a(\pi) \mathbf{h}^a \\
& = 
\sum_{a} \mathbf{x}_p^a(\pi) \big[(\mathbf{I} - \beta \mathbf{P}^a) \mathbf{F}(S) - \mathbf{h}^a\big] - F_p(S) + 
F_p(\pi) \\
& = 
\sum_{a' \neq a} \mathbf{x}_{p S_{a'}}^a(\pi) \mathbf{f}_{S_{a'}}^{a', a}(S) - F_p(S) + 
F_p(\pi)  \\
& = 
\sum_{a' > a} \mathbf{x}_{p S_{a'}}^a(\pi) \mathbf{f}_{S_{a'}}^{a', a}(S) + \sum_{a' < a} \mathbf{x}_{p S_{a'}}^a(\pi) \mathbf{f}_{S_{a'}}^{a', a}(S)  - F_p(S) + 
F_p(\pi) \\
& = 
\sum_{a' < a} \mathbf{x}_{p S_{a'}}^a(\pi) \mathbf{f}_{S_{a'}}^{a', a}(S) -\sum_{a' > a} \mathbf{x}_{p S_{a'}}^a(\pi) \mathbf{f}_{S_{a'}}^{a, a'}(S)  - F_p(S) + 
F_p(\pi) \\
& = \sum_{a' < a} \sum_{j \in S_{a'}} x_{p j}^a(\pi) f_{j}^{a', a}(S) -\sum_{a' > a} \sum_{j \in S_{a'}} x_{p j}^a(\pi) f_{j}^{a, a'}(S)  - F_p(S) + 
F_p(\pi).
\end{align*}

Part (b) follows similarly as part (a).
\end{proof}

The following result draws on the above to relate metrics $F_p(S)$ and $F_p(\mathcal{T}_j^{a, a'} S)$ to $G_p(S)$ and $G_p(\mathcal{T}_j^{a, a'} S)$, respectively, which clarifies the interpretation of marginal metrics $f_{j}^{a, a'}(S)$ and $g_{j}^{a, a'}(S)$. Recall that $\mathcal{T}_j^{a, a'} S$ is the modification of policy $S$ that results 
by shifting gear in state $j$ from $a$ to $a'$.

\begin{lemma}
\label{lma:FfGgkl} For any actions $a \neq a'$ and state $j \in S_a$:
\begin{itemize}
\item[(a)] $F_p(S) =   
F_p(\mathcal{T}_j^{a, a'} S) + f_{j}^{a, a'}(S) x_{pj}^{a'}(\mathcal{T}_j^{a, a'} S);$
\item[(b)] $F_p(\mathcal{T}_j^{a, a'} S) =   
F_p(S) + f_{j}^{a', a}(\mathcal{T}_j^{a, a'} S) x_{pj}^{a}(S);$
\item[(c)] $G_p({\mathcal{T}_j^{a, a'} S}) = G_p(S) + g_{j}^{a, a'}(S) x_{pj}^{a'}(\mathcal{T}_j^{a, a'} S);$
\item[(d)] $G_p(S) = G_p(\mathcal{T}_j^{a, a'} S) + g_{j}^{a', a}(\mathcal{T}_j^{a, a'} S) x_{pj}^{a}(S).$
\end{itemize}
\end{lemma}
\begin{proof}
(a) Writing Lemma \ref{lma:declaws} (a) as 
\[
F_p(S) =   
F_p(\pi) + \sum_{l \neq k} \sum_{j' \in S_{l}} f_{j'}^{l, k}(S) x_{pj'}^k(\pi)   
\]
and taking $\pi = \mathcal{T}_j^{a, a'} S$ in the latter  expression, with $j \in S_a$, gives
\begin{equation}
\label{eq:FSFTjaapS}
F_p(S) =   
F_p(\mathcal{T}_j^{a, a'} S) + \sum_{l \neq k} \sum_{j' \in S_{l}} f_{j'}^{l, k}(S) x_{pj'}^k(\mathcal{T}_j^{a, a'} S).
\end{equation}

Now, on the one hand, for $l \neq a$, policy $\mathcal{T}_j^{a, a'} S$ selects gear $l$ in states $j' \in S_l$. Hence, for $k \neq l$ and such $j'$, 
$x_{p j'}^{k}(\mathcal{T}_j^{a, a'} S) = 0$.

On the other hand, for $l = a$, policy $\pi = \mathcal{T}_j^{a, a'} S$ selects gear $l = a$ in states $j' \in S_l - \{j\}$ and selects gear $a'$ in state $j' = j$. Hence, $x_{pj'}^{k}(\mathcal{T}_j^{a, a'} S) = 0$ for $j' \in S_l - \{j\}$, since $k \neq l = a$ and $x_{p j'}^{k}(\mathcal{T}_j^{a, a'} S) = 0$ for $j' = j$ if $k \neq a'$. 

Thus, the only  positive $x_{p j'}^k(\mathcal{T}_j^{a, a'} S)$ in (\ref{eq:FSFTjaapS}) can be $x_{pj}^{a'}(\mathcal{T}_j^{a, a'} S)$; hence, (\ref{eq:FSFTjaapS}) reduces to
\[
F_p(S) =   
F_p(\mathcal{T}_j^{a, a'} S) + f_{j}^{a, a'}(S) x_{pj}^{a'}(\mathcal{T}_j^{a, a'} S).
\]

(b) Let $S' = \mathcal{T}_j^{a, a'} S$ and note that $j \in S_{a'}'$ and $S = \mathcal{T}_j^{a', a} S'$. Hence, by part (a), 
\begin{align*}
F_p(\mathcal{T}_j^{a, a'} S)  & =  F_p(S') =  
F_p(\mathcal{T}_j^{a', a} S') + f_{j}^{a', a}(S') x_{pj}^{a}(\mathcal{T}_j^{a', a} S') \\
& = F_p(S) + f_{j}^{a', a}(\mathcal{T}_j^{a, a'} S) x_{pj}^{a}(S).
\end{align*}

Parts (c) and (d) follow similarly as (a) and (b).
\end{proof}

The following result, which follows easily from Lemma \ref{lma:FfGgkl} (c, d), clarifies the interpretation of PCL$(\mathcal{F})$-indexability condition (PCLI1) in Definition \ref{def:pclind}, as a natural monotonicity property of the resource metric $G_i(S)$.

\begin{proposition}
\label{pro:pcli1eq}
Condition \textup{(PCLI1)} is equivalent to the following: for $S \in \mathcal{F}, a' < a < a'', j \in S_a$,
\begin{itemize}
\item[(a)] 
$
G_i(\mathcal{T}_j^{a, a'} S) \leqslant G_i(S) \leqslant G_i({\mathcal{T}_j^{a, a''} S}), \enspace i \neq j, \enspace
$
and \enspace
$
G_j(\mathcal{T}_j^{a, a'} S) < G_j(S) < G_j({\mathcal{T}_j^{a, a''} S}).
$
\item[(b)] If $p$ has full support, 
$
G_p(\mathcal{T}_j^{a, a'} S) < G_p(S) < G_p({\mathcal{T}_j^{a, a''} S}).
$
\end{itemize}
\end{proposition}

\begin{remark}
\label{re:pcli1eq} {Proposition \ref{pro:pcli1eq} yields the following intuitive interpretation of condition (PCLI1): it means that, for any $\mathcal{F}$-policy $S$, modifying $S$ by downshifting gears in one state results in a lower or equal value of the resource usage metric; and, conversely, modifying $S$ by upshifting gears in one state results in a higher or equal resource usage metric. When the initial state is drawn from a distribution with full support, downshifting gears leads to a strictly lower resource usage, and upshifting gears leads to a strictly higher resource usage.}
\end{remark}

The next result shows that, under (PCLI1), the increment $F_p(\mathcal{T}_j^{a, a'} S) - F_p(S)$ is proportional to $G_p(S) - G_p({\mathcal{T}_j^{a, a'} S})$, with proportionality constant $m_{j}^{a', a}(S)$, 
which is equal to $m_{j}^{a', a}(\mathcal{T}_j^{a, a'} S)$.

\begin{lemma}
\label{lma:P64MP02} Under condition \textup{(PCLI1)}, for any 
 actions $a' \neq a$ and state $j \in S_a,$
\begin{itemize}
\item[(a)]\small{$F_p(\mathcal{T}_j^{a, a'} S) - F_p(S) = m_{j}^{a', a}(S) (G_p(S) - G_p({\mathcal{T}_j^{a, a'} S})) = m_{j}^{a', a}(\mathcal{T}_j^{a, a'} S) (G_p(S) - G_p({\mathcal{T}_j^{a, a'} S}));$}
\item[(b)] $m_{j}^{a', a}(S) =  m_{j}^{a', a}(\mathcal{T}_j^{a, a'} S).$
\end{itemize}
\end{lemma}
\begin{proof}
(a) We have, using Lemma \ref{lma:FfGgkl} (a, c), 
\begin{align*}
F_p(\mathcal{T}_j^{a, a'} S) - F_p(S) & =  -f_{j}^{a, a'}(S) x_{pj}^{a'}(\mathcal{T}_j^{a, a'} S) = 
-f_{j}^{a, a'}(S) \frac{G_p({\mathcal{T}_j^{a, a'} S}) - G_p(S)}{g_{j}^{a, a'}(S)} \\
& = m_{j}^{a', a}(S) (G_p(S) - G_p({\mathcal{T}_j^{a, a'} S})).
\end{align*}

On the other hand, using Lemma \ref{lma:FfGgkl} (b, d) we obtain 
\begin{align*}
F_p(\mathcal{T}_j^{a, a'} S) - F_p(S) & =  f_{j}^{a', a}(\mathcal{T}_j^{a, a'} S) x_{pj}^{a}(S) = 
f_{j}^{a', a}(\mathcal{T}_j^{a, a'} S) \frac{G_p(S) - G_p(\mathcal{T}_j^{a, a'} S)}{g_{j}^{a', a}(\mathcal{T}_j^{a, a'} S)} \\
& = m_{j}^{a', a}(\mathcal{T}_j^{a, a'} S) (G_p(S) - G_p({\mathcal{T}_j^{a, a'} S})).
\end{align*}

(b) This part follows directly from part (a).
\end{proof}

The following result shows in its part (a) that, under condition (PCLI1), the increment $f_j^{a', a}(S) - f_j^{a', a}(\mathcal{T}_i^{\bar{a}, \bar{a}'} S)$ is proportional to $g_j^{a', a}(S) - g_j^{a', a}(\mathcal{T}_i^{\bar{a}, \bar{a}'} S)$, with the proportionality constant being $m_{i}^{\bar{a}', \bar{a}}(S)$.
Then, its part (b) draws on this result to obtain a relation between MP metrics that we use in Algorithm \ref{alg:tdabutd}.

\begin{lemma}
\label{lma:P64MP02more} Under condition \textup{(PCLI1)}, for any actions 
$a' \neq a$ and $\bar{a}' \neq \bar{a}$ with 
$\mathcal{T}_i^{\bar{a}, \bar{a}'} S \in \mathcal{F}$ and states $i \in S_{\bar{a}}$ and $j,$
\begin{itemize}
\item[(a)] $f_j^{a', a}(S) - f_j^{a', a}(\mathcal{T}_i^{\bar{a}, \bar{a}'} S) = m_{i}^{\bar{a}', \bar{a}}(S) (g_j^{a', a}(S) - g_j^{a', a}(\mathcal{T}_i^{\bar{a}, \bar{a}'} S));$
\item[(b)] $g_j^{a', a}(\mathcal{T}_i^{\bar{a}, \bar{a}'} S) \big(m_j^{a', a}(\mathcal{T}_i^{\bar{a}, \bar{a}'} S)
 - m_{i}^{\bar{a}', \bar{a}}(S)\big) = 
 g_j^{a', a}(S) \big(m_j^{a', a}(S)  - m_{i}^{\bar{a}', \bar{a}}(S)\big).$
\end{itemize}
\end{lemma}
\begin{proof}
(a) From (\ref{eq:vfiaapS}) and (\ref{eq:vgiaapS}), we obtain
\[
\mathbf{f}^{a', a}(S) - \mathbf{f}^{a', a}(\mathcal{T}_i^{\bar{a}, \bar{a}'} S) = \beta (\mathbf{P}^{a'} - \mathbf{P}^{a}\big) \big(\mathbf{F}(S) - \mathbf{F}(\mathcal{T}_i^{\bar{a}, \bar{a}'} S))
\]
and
\[
\mathbf{g}^{a', a}(S) - \mathbf{g}^{a', a}(\mathcal{T}_i^{\bar{a}, \bar{a}'} S) = \beta \big(\mathbf{P}^{a} - \mathbf{P}^{a'}\big) (\mathbf{G}(S) - \mathbf{G}(\mathcal{T}_i^{\bar{a}, \bar{a}'} S))
\]

Now, combining Lemma \ref{lma:P64MP02}(a) with this gives
\[
\mathbf{f}^{a', a}(S) - \mathbf{f}^{a', a}(\mathcal{T}_i^{\bar{a}, \bar{a}'} S) = m_{i}^{\bar{a}', \bar{a}}(S) (\mathbf{g}^{a', a}(S) - \mathbf{g}^{a', a}(\mathcal{T}_i^{\bar{a}, \bar{a}'} S)).
\]

(b) We can write
\begin{align*}
m_j^{a', a}(\mathcal{T}_i^{\bar{a}, \bar{a}'} S) & \triangleq 
\frac{f_j^{a', a}(\mathcal{T}_i^{\bar{a}, \bar{a}'} S)}{g_j^{a', a}(\mathcal{T}_i^{\bar{a}, \bar{a}'} S)} \\
 & = \frac{f_j^{a', a}(S)}{g_j^{a', a}(\mathcal{T}_i^{\bar{a}, \bar{a}'} S)} - m_{i}^{\bar{a}', \bar{a}}(S)  \bigg(\frac{g_j^{a', a}(S)}{g_j^{a', a}(\mathcal{T}_i^{\bar{a}, \bar{a}'} S)} - 1\bigg) \\
& = \frac{g_j^{a', a}(S)}{g_j^{a', a}(\mathcal{T}_i^{\bar{a}, \bar{a}'} S)} \frac{f_j^{a', a}(S)}{g_j^{a', a}(S)}  - m_{i}^{\bar{a}', \bar{a}}(S)  \bigg(\frac{g_j^{a', a}(S)}{g_j^{a', a}(\mathcal{T}_i^{\bar{a}, \bar{a}'} S)} - 1\bigg) \\
& = \frac{g_j^{a', a}(S)}{g_j^{a', a}(\mathcal{T}_i^{\bar{a}, \bar{a}'} S)} m_j^{a', a}(S)  - m_{i}^{\bar{a}', \bar{a}}(S)  \bigg(\frac{g_j^{a', a}(S)}{g_j^{a', a}(\mathcal{T}_i^{\bar{a}, \bar{a}'} S)} - 1\bigg) \\
& = m_{i}^{\bar{a}', \bar{a}}(S) + 
 \frac{g_j^{a', a}(S)}{g_j^{a', a}(\mathcal{T}_i^{\bar{a}, \bar{a}'} S)} \big(m_j^{a', a}(S)  - m_{i}^{\bar{a}', \bar{a}}(S)\big),
\end{align*}
where the second line is obtained by using part (a).
\end{proof}

\section{Analysis of the Output of Algorithm $\mathrm{DS}(\mathcal{F})$}
\label{s:aoadsf}
This section derives further relations between project performance metrics that will play the role of key tools for elucidating and analyzing the output of the downshift adaptive-greedy index algorithm $\mathrm{DS}(\mathcal{F})$ in Algorithm \ref{alg:tdabutd} and hence to prove the verification theorem.
Throughout this section, $\{(j_k, a_k), S^k, m_{j_k}^{*, a_k}\}_{k=1}^{K}$
is an output of algorithm $\mathrm{DS}(\mathcal{F})$.

We start with a result justifying recursive index update Formula (\ref{eq:recindcpt}) in the algorithm.
\begin{lemma}
\label{lma:recindcomp} Let condition \textup{(PCLI1)} hold.  
Then, for any state $j$, active action $a \geqslant 1$, and $k = 2, \ldots, K$, 
\begin{equation}
\label{eq:recindcomp}
m_j^{a-1, a}(S^{k}) = m_{j_{k-1}}^{*, a_{k-1}} + 
 \frac{g_j^{a-1, a}(S^{k-1})}{g_j^{a-1, a}(S^k)} (m_j^{a-1, a}(S^{k-1})  - m_{j_{k-1}}^{*, a_{k-1}}),
\end{equation}
or, equivalently,
\begin{equation}
\label{eq:eqrecindcomp}
m_j^{a-1, a}(S^{k-1})  = m_{j_{k-1}}^{*, a_{k-1}} + 
\frac{g_j^{a-1, a}(S^k)}{g_j^{a-1, a}(S^{k-1})} \big(m_j^{a-1, a}(S^{k}) - m_{j_{k-1}}^{*, a_{k-1}}\big).
\end{equation}
\end{lemma}
\begin{proof}
The result follows by noting that $m_{j_{k-1}}^{*, a_{k-1}} = m_{j_{k-1}}^{a_{k-1}-1, a_{k-1}}$ and 
taking $i = j_{k-1}$, $\bar{a}' = a_{k-1}-1$, $\bar{a} = a_{k-1}$, $S = S^{k-1}$, and $a' = a-1$ in Lemma
\ref{lma:P64MP02more}(b), since  
$\mathcal{T}_i^{\bar{a}, \bar{a}'} S = S^k \in \mathcal{F}$.
\end{proof}

The following result expresses, in its part (b), the MP metric $m_{j_l}^{a_l-1, a_l}(S^k)$, for \mbox{$k < l$}, as the sum of index $m_{j_k}^{*,  a_k}$ and a positive linear combination of the  index differences $m_{j_n}^{*, a_n}~-~m_{j_{n-1}}^{*, a_{n-1}}$ for $n=k+1, \ldots, l$.
Part (a) is a preliminary result needed to prove part (b).

\begin{lemma}
\label{lma:mjlalm1alsklprim} Under condition \textup{(PCLI1)},   the following holds for $1 \leqslant k < l \leqslant K$:
\begin{itemize}
\item[(a)] For $l' = k+1, \ldots, l,$
\begin{equation}
\label{eq:mjlalm1alsklprim}
\begin{split}
m_{j_l}^{a_l-1, a_l}(S^k) & = 
m_{j_k}^{*,  a_k} + \frac{1}{g_{j_l}^{a_l-1, a_l}(S^k)} \bigg[\sum_{n=k+1}^{l'-1} g_{j_l}^{a_l-1, a_l}(S^n) (m_{j_n}^{*, a_n}-m_{j_{n-1}}^{*, a_{n-1}}) \bigg. \\
& \qquad \bigg. + g_{j_l}^{a_l-1, a_l}(S^{l'}) \big(m_{j_l}^{a_l-1, a_l}(S^{l'}) - m_{j_{l'-1}}^{*, a_{l'-1}}\big)\bigg].
\end{split}
\end{equation}
\item[(b)] $f_{j_l}^{a_l-1, a_l}(S^k) = 
g_{j_l}^{a_l-1, a_l}(S^k) m_{j_k}^{*,  a_k} +  \sum_{n=k+1}^{l}  g_{j_l}^{a_l-1, a_l}(S^n) (m_{j_n}^{*, a_n}-m_{j_{n-1}}^{*, a_{n-1}})$, or, equivalently, 
\begin{equation}
\label{eq:2mjlalm1alsklprim}
m_{j_l}^{a_l-1, a_l}(S^k) = 
m_{j_k}^{*,  a_k} +  \sum_{n=k+1}^{l}  \frac{g_{j_l}^{a_l-1, a_l}(S^n)}{g_{j_l}^{a_l-1, a_l}(S^k)} (m_{j_n}^{*, a_n}-m_{j_{n-1}}^{*, a_{n-1}}).
\end{equation}
\end{itemize}
\end{lemma}
\begin{proof}
(a) We prove the result by induction on $l' = k+1, \ldots, l$.
For $l' = k+1$, Equation~(\ref{eq:mjlalm1alsklprim}) holds because, by Equation~(\ref{eq:eqrecindcomp}) in Lemma \ref{lma:recindcomp}, we have
\begin{align*}
m_{j_l}^{a_l-1, a_l}(S^k) & = 
m_{j_k}^{*,  a_k} + \frac{g_{j_l}^{a_l-1, a_l}(S^{k+1})}{g_{j_l}^{a_l-1, a_l}(S^k)}  \big(m_{j_l}^{a_l-1, a_l}(S^{k+1}) - m_{j_{k}}^{*, a_{k}}\big).
\end{align*}

Suppose now that Equation~(\ref{eq:mjlalm1alsklprim}) holds for some $l'$ with $k < l' < l$.
We will use that, by Equation~(\ref{eq:eqrecindcomp}) in Lemma \ref{lma:recindcomp}, we have
\begin{align*}
m_{j_l}^{a_l-1, a_l}(S^{l'}) & = 
m_{j_{l'}}^{*, a_{l'}} + \frac{g_{j_l}^{a_l-1, a_l}(S^{l'+1})}{g_{j_l}^{a_l-1, a_l}(S^{l'})}  \big(m_{j_l}^{a_l-1, a_l}(S^{l'+1}) - m_{j_{l'}}^{*, a_{l'}}\big).
\end{align*}

Now, substituting the right-hand side of the latter identity for $m_{j_l}^{a_l-1, a_l}(S^{l'})$ in Equation~(\ref{eq:mjlalm1alsklprim}) gives
\begin{multline*}
m_{j_l}^{a_l-1, a_l}(S^k) = 
m_{j_k}^{*,  a_k} + \frac{1}{g_{j_l}^{a_l-1, a_l}(S^k)} \bigg[\sum_{n=k+1}^{l'-1} g_{j_l}^{a_l-1, a_l}(S^n) (m_{j_n}^{*, a_n}-m_{j_{n-1}}^{*, a_{n-1}}) \bigg. \\
 \bigg. + g_{j_l}^{a_l-1, a_l}(S^{l'}) \bigg(m_{j_{l'}}^{*, a_{l'}} + \frac{g_{j_l}^{a_l-1, a_l}(S^{l'+1})}{g_{j_l}^{a_l-1, a_l}(S^{l'})}  \big(m_{j_l}^{a_l-1, a_l}(S^{l'+1}) - m_{j_{l'}}^{*, a_{l'}}\big) - m_{j_{l'-1}}^{*, a_{l'-1}}\bigg)\bigg]  \\
= m_{j_k}^{*,  a_k} + \frac{1}{g_{j_l}^{a_l-1, a_l}(S^k)} \bigg[\sum_{n=k+1}^{l'} g_{j_l}^{a_l-1, a_l}(S^n) (m_{j_n}^{*, a_n}-m_{j_{n-1}}^{*, a_{n-1}}) \bigg. \\
 \bigg. +  g_{j_l}^{a_l-1, a_l}(S^{l'+1})  \big(m_{j_l}^{a_l-1, a_l}(S^{l'+1}) - m_{j_{l'}}^{*, a_{l'}}\big)\bigg],
 \end{multline*}
 which shows that the result also holds for $l'+1$ and hence completes the induction.
 
 (b) This part corresponds to the case 
 $l' = l$ in part (a), noting that $m_{j_l}^{a_l-1, a_l}(S^{l}) = m_{j_{l}}^{*, a_{l}}$.
 \end{proof}
 
In the following result and henceforth we use the notation $a^k(j)$ to denote the action selected in state $j$ by policy $S^k$. 
Thus, e.g., $a^1(j) = A$ for every state $j$, since $S^1 = (\emptyset, \ldots, \emptyset, \mathcal{N})$, and 
$a^{K+1}(j) = 0$ for every state $j$, since $S^{K+1} = (\mathcal{N}, \emptyset, \ldots, \emptyset)$.

\begin{lemma}
\label{lma:3mjlalm1alsklprim} Under condition \textup{(PCLI1)},   the following holds for $1 \leqslant k \leqslant l \leqslant K$:
\begin{equation}
\label{eq:3mjlalm1alsklprim}
f_{j_l}^{a_l-1, a^k(j_l)}(S^k)  = 
g_{j_l}^{a_l-1, a^k(j_l)}(S^k) m_{j_k}^{*,  a_k}  + \sum_{n=k+1}^{l} g_{j_l}^{a_l-1, a^n(j_l)}(S^n) (m_{j_n}^{*, a_n}-m_{j_{n-1}}^{*, a_{n-1}}),
\end{equation}
or, equivalently, 
\[
m_{j_l}^{a_l-1, a^k(j_l)}(S^k)  = 
m_{j_k}^{*,  a_k}  + \sum_{n=k+1}^{l} \frac{g_{j_l}^{a_l-1, a^n(j_l)}(S^n) }{g_{j_l}^{a_l-1, a^k(j_l)}(S^k) }(m_{j_n}^{*, a_n}-m_{j_{n-1}}^{*, a_{n-1}}).
\]
\end{lemma}
\begin{proof}
\textls[-25]{Fix $k$. We prove the result by induction on $l = k, \ldots, K$. For $l = k$, using that \mbox{$a^k(j_k) = a_k$}}  gives that (\ref{eq:3mjlalm1alsklprim})  reduces to 
$f_{j_l}^{a_k-1, a_k}(S^k)  = 
g_{j_k}^{a_k-1, a_k}(S^k) m_{j_k}^{*,  a_k}$, which holds by construction, since $m_{j_k}^{*,  a_k}$ is defined in the algorithm precisely as $m_{j_k}^{a_k-1, a_k}(S^k) = f_{j_l}^{a_k-1, a_k}(S^k)/ 
g_{j_k}^{a_k-1, a_k}(S^k)$.

Suppose now that (\ref{eq:3mjlalm1alsklprim}) holds up to and including some $l$ with $k < l < K$. 
We will prove that it must then hold for $l+1$. For such a purpose, we distinguish two cases, depending on whether $j_{l+1} = j_l$ or $j_{l+1} \neq j_l$.
Start with the case $j_{l+1} = j_l$. To simplify the argument below, we write $j_l$ as $j$ and $a_l$ as $a$. In this case, the algorithm downshifts in step $l$ the gear  in state $j$ from $a$ to $a-1$ and in step $l+1$ downshifts again in state $j$ from gear $a-1$ to $a-2$. Hence, $a^l(j) = a$ and $a^{l+1}(j) = a_{l+1} = a-1$. We can write
\begin{align*}
f_{j_{l+1}}^{a_{l+1}-1, a^k(j_{l+1})}(S^k)  & = f_{j_{l+1}}^{a_{l+1}-1, a_{l+1}}(S^k) + f_{j_{l+1}}^{a_{l+1}, a^k(j)}(S^k) = 
f_{j}^{a-2, a-1}(S^k) + f_{j}^{a-1, a^k(j)}(S^k)
\\
& =  
g_{j}^{a-2, a-1}(S^k) m_{j_k}^{*,  a_k} +  \sum_{n=k+1}^{l+1}  g_{j}^{a-2, a-1}(S^n) (m_{j_n}^{*, a_n}-m_{j_{n-1}}^{*, a_{n-1}}) \\
& \quad +  
g_{j}^{a-1, a^k(j)}(S^k) m_{j_k}^{*,  a_k}  + \sum_{n=k+1}^{l} g_{j}^{a-1, a^n(j)}(S^n) (m_{j_n}^{*, a_n}-m_{j_{n-1}}^{*, a_{n-1}}) \\
& =  
g_{j}^{a-2, a^k(j)}(S^k) m_{j_k}^{*,  a_k} +  \sum_{n=k+1}^{l+1}  g_{j}^{a-2, a^n(j)}(S^n) (m_{j_n}^{*, a_n}-m_{j_{n-1}}^{*, a_{n-1}}) \\
& = g_{j_{l+1}}^{a_{l+1}-1, a^k(j)}(S^k) m_{j_k}^{*,  a_k} +  \sum_{n=k+1}^{l+1}  g_{j_{l+1}}^{a_{l+1}-1, a^n(j)}(S^n) (m_{j_n}^{*, a_n}-m_{j_{n-1}}^{*, a_{n-1}}),
\end{align*}
where we have used in turn the elementary property $f_{j}^{a, a''}(S)  = f_{j}^{a, a'}(S) + f_{j}^{a', a''}(S)$, \mbox{Lemma \ref{lma:mjlalm1alsklprim} (b)}, the induction hypothesis, and $a^{l+1}(j) = a-1$. Therefore, the result holds for $l+1$ in this case.

Consider now the case $j_{l+1} \neq j_l$, in which $a_{l+1} = a^{l+1}(j_{l+1}) = a^{l}(j_{l+1})$. To simplify the argument below, we write $j_{l+1}$ as $j$ and $a_{l+1}$ as $a$. 
In this case, the algorithm downshifts in step $l+1$ at state $j$ from gear $a$ to $a-1$.
If $a < A$, the previous downshift at state $j$, from gear $a+1$ to $a$,  occurred at some earlier step $l' < l$, so $j_{l'} = j$, $a_{l'}-1 = a$, and 
\begin{equation}
\label{eq:alprimeid}
a^{l'+1}(j)  = \cdots = a^{l}(j) = a^{l+1}(j) = a.
\end{equation}

We can now write
\begin{align*}
f_{j_{l+1}}^{a_{l+1}-1, a^k(j_{l+1})}(S^k)  & = f_{j_{l+1}}^{a_{l+1}-1, a_{l+1}}(S^k) + f_{j_{l+1}}^{a_{l+1}, a^k(j)}(S^k) = 
f_{j}^{a-1, a}(S^k) + f_{j}^{a, a^k(j)}(S^k) \\
& = g_{j}^{a-1, a}(S^k) m_{j_k}^{*,  a_k}  + \sum_{n=k+1}^{l+1} g_{j}^{a-1, a}(S^n) (m_{j_n}^{*, a_n}-m_{j_{n-1}}^{*, a_{n-1}}) \\
& \quad + g_{j_{l'}}^{a_{l'}-1, a^k(j_{l'})}(S^k) m_{j_k}^{*,  a_k}  + \sum_{n=k+1}^{l'} g_{j_l}^{a_{l'}-1, a^n(j_{l'})}(S^n) (m_{j_n}^{*, a_n}-m_{j_{n-1}}^{*, a_{n-1}}) \\
& = g_{j}^{a-1, a^k(j)}(S^k) m_{j_k}^{*,  a_k}  + \sum_{n=k+1}^{l'} g_{j}^{a-1, a^n(j)}(S^n) (m_{j_n}^{*, a_n}-m_{j_{n-1}}^{*, a_{n-1}}) \\
& \quad + \sum_{n=l'+1}^{l+1} g_{j}^{a-1, a}(S^n) (m_{j_n}^{*, a_n}-m_{j_{n-1}}^{*, a_{n-1}}) \\
& = g_{j_{l+1}}^{a_{l+1}-1, a^k(j)}(S^k) m_{j_k}^{*,  a_k} +  \sum_{n=k+1}^{l+1}  g_{j_{l+1}}^{a_{l+1}-1, a^n(j)}(S^n) (m_{j_n}^{*, a_n}-m_{j_{n-1}}^{*, a_{n-1}}),
\end{align*}
where we have used in turn Lemma \ref{lma:mjlalm1alsklprim} (b), the induction hypothesis, and (\ref{eq:alprimeid}).

Suppose now that $a = A$. Then,
\begin{equation}
\label{eq:2alprimeid}
a^k(j) = a^{k+1}(j)  = \cdots = a^{l+1}(j) = A,
\end{equation}
and we can write 
\begin{align*}
f_{j_{l+1}}^{a_{l+1}-1, a^k(j_{l+1})}(S^k)  & = 
f_{j}^{A-1, A}(S^k)  \\
& = g_{j}^{A-1, A}(S^k) m_{j_k}^{*,  a_k}  + \sum_{n=k+1}^{l+1} g_{j}^{A-1, A}(S^n) (m_{j_n}^{*, a_n}-m_{j_{n-1}}^{*, a_{n-1}}) \\
& = g_{j_{l+1}}^{a_{l+1}-1, a^k(j)}(S^k) m_{j_k}^{*,  a_k} +  \sum_{n=k+1}^{l+1}  g_{j_{l+1}}^{a_{l+1}-1, a^n(j)}(S^n) (m_{j_n}^{*, a_n}-m_{j_{n-1}}^{*, a_{n-1}}),
\end{align*}
where we have used in turn Lemma \ref{lma:mjlalm1alsklprim} (b) and (\ref{eq:2alprimeid}).

Hence, the result also holds for $l+1$ in the case $j_{l+1} \neq j_l$, which completes the induction proof.
\end{proof}

The next result expresses the modified holding costs $\widehat{h}_{j}^{a}$ as positive linear combinations of $m_{j_1}^{*,  a_1}$ and the differences $m_{j_n}^{*, a_n}-m_{j_{n-1}}^{*, a_{n-1}}$. 
We will use it in Lemmas \ref{lma:objdual2} and \ref{lma:objduallamb} to reformulate  the modified cost objective $\widehat{V}_p(\lambda, \pi)$.

\begin{lemma}
\label{lma:objdual}  Under condition (PCLI1),
\begin{equation}
\label{eq:hhatgy}
\begin{split}
g_{j_1}^{a_1-1, A}(S^1) m_{j_1}^{*,  a_1} & = \widehat{h}_{j_1}^{a_1-1} \\
g_{j_l}^{a_l-1, A}(S^1) m_{j_1}^{*,  a_1}  + \sum_{n=2}^{l} g_{j_l}^{a_l-1, a^n(j_l)}(S^n) (m_{j_n}^{*, a_n}-m_{j_{n-1}}^{*, a_{n-1}}) & = \widehat{h}_{j_l}^{a_l-1}, \enspace l = 2, \ldots, K.
\end{split}
\end{equation}
\end{lemma}
\begin{proof}
The result follows from Corollary \ref{cor:fiam1as1hiam1} (a), which shows that  $f_j^{a-1, A}(S^1) = \widehat{h}_j^{a-1} - \widehat{h}_j^{A} = \widehat{h}_j^{a-1}$ (since $\widehat{h}_j^{A}  = 0$), and Lemma \ref{lma:3mjlalm1alsklprim}, used with $k = 1$, noting that $a^1(j) = A$.
\end{proof}

The following result shows that $m_{{j_{k-1}}}^{*, a_{k-1}}$ can be expressed in two different ways in terms of MP metrics. 
\begin{lemma}
\label{lma:indextwoway}  Under condition (PCLI1), 
\[
m_{{j_{k-1}}}^{a_{k-1}-1, a_{k-1}}(S^k) = m_{{j_{k-1}}}^{a_{k-1}-1, a_{k-1}}(S^{k-1}) = m_{{j_{k-1}}}^{*, a_{k-1}}, \enspace k = 2, \ldots, K+1.
\]
\end{lemma}
\begin{proof}
The first identity follows from the result $m_{j}^{a', a}(\mathcal{T}_j^{a, a'} S) = m_{j}^{a', a}(S)$ in Lemma \ref{lma:P64MP02} (b), taking $S = S^{k-1}$, $a = a_{k-1}$, $a' = a_{k-1}-1$, and $j = j_{k-1}$,  and noting that $ S^k = \mathcal{T}_{j_{k-1}}^{a_{k-1}, a_{k-1}-1} S^{k-1}$. The second identity follows by definition of $m_{{j_{k-1}}}^{*, a_{k-1}}$.
\end{proof}

The following result relates metrics under two successive policies as generated by the index algorithm.
\begin{lemma}
\label{lma:lambdapprec}  Under condition (PCLI1), 
\begin{itemize}
\item[(a)] $F_p(S^k) =   
F_p(S^{k-1}) + f_{{j_{k-1}}}^{a_{k-1}-1, a_{k-1}}(S^k) x_{p {j_{k-1}}}^{a_{k-1}}(S^{k-1}),$ \enspace $k = 2, \ldots, K+1;$
\item[(b)] $F_p(S^{k})  = F_p(S^{k+1}) - f_{{j_{k}}}^{a_{k}-1, a_{k}}(S^{k+1}) x_{p {j_{k}}}^{a_{k}}(S^{k}),$ \enspace $k = 1, \ldots, K;$
\item[(c)] $G_p(S^k)  = G_p(S^{k-1}) - g_{j_{k-1}}^{a_{k-1}-1, a_{k-1}}(S^k) x_{p {j_{k-1}}}^{a_{k-1}}(S^{k-1}),$ \enspace $k = 2, \ldots, K+1;$
\item[(d)] $G_p(S^{k})  = G_p(S^{k+1}) + g_{j_{k}}^{a_{k}-1, a_{k}}(S^{k+1}) x_{p {j_{k}}}^{a_{k}}(S^{k}),$ \enspace $k = 1, \ldots, K;$
\item[(e)] $V_p(\lambda, S^k) = V_p(\lambda, S^{k-1}) - (\lambda - m_{j_{k-1}}^{*, a_{k-1}}) 
  g_{j_{k-1}}^{a_{k-1}-1, a_{k-1}}(S^k) x_{p j_{k-1}}^{a_{k-1}}(S^{k-1})$, \enspace k = 2, \ldots, K+1.
\item[(f)] $V_p(\lambda, S^k) = V_p(\lambda, S^{k+1}) - (m_{j_{k}}^{*, a_{k}} - \lambda) 
  g_{j_{k}}^{a_{k}-1, a_{k}}(S^{k+1}) x_{p j_{k}}^{a_{k}}(S^{k})$, \enspace k = 1, \ldots, K.
\end{itemize}
\end{lemma}
\begin{proof}
(a) This part follows from Lemma \ref{lma:FfGgkl} (a) by taking $S = S^k$, $a = a_{k-1}-1$, $a' = a_{k-1}$, and $j = j_{k-1}$,  noting that $\mathcal{T}_{j_{k-1}}^{a_{k-1}-1, a_{k-1}} S^k = S^{k-1}$ in $F_p(S) =   
F_p(\mathcal{T}_j^{a, a'} S) + f_{j}^{a, a'}(S) x_{pj}^{a'}(\mathcal{T}_j^{a, a'} S)$. 

(b) This part follows directly from (a).

(c) This part follows from Lemma \ref{lma:FfGgkl} (c) by taking $S$, $a$, $a'$, and $j$ as in part (a) in $G_p({\mathcal{T}_j^{a, a'} S}) = G_p(S) + g_{j}^{a, a'}(S) x_{pj}^{a'}(\mathcal{T}_j^{a, a'} S)$. 

(d) This part follows directly from (c).

(e) The result follows from 
\begin{align*}
V_p(\lambda, S^{k}) & = F_p(S^k) + \lambda G_p(S^k) \\
& = F_p(S^{k-1}) + f_{{j_{k-1}}}^{a_{k-1}-1, a_{k-1}}(S^k) x_{p {j_{k-1}}}^{a_{k-1}}(S^{k-1}) \\
& \qquad + \lambda \big[G_p(S^{k-1}) - g_{j_{k-1}}^{a_{k-1}-1, a_{k-1}}(S^k) x_{p {j_{k-1}}}^{a_{k-1}}(S^{k-1})\big] \\
& = V_p(\lambda, S^{k-1}) + \big(f_{{j_{k-1}}}^{a_{k-1}-1, a_{k-1}}(S^k) - \lambda g_{j_{k-1}}^{a_{k-1}-1, a_{k-1}}(S^k)\big)
x_{p {j_{k-1}}}^{a_{k-1}}(S^{k-1}) \\
& = V_p(\lambda, S^{k-1}) + \big(m_{{j_{k-1}}}^{a_{k-1}-1, a_{k-1}}(S^k) - \lambda\big) g_{j_{k-1}}^{a_{k-1}-1, a_{k-1}}(S^k)
x_{p {j_{k-1}}}^{a_{k-1}}(S^{k-1}) \\
& = V_p(\lambda, S^{k-1}) - \big(\lambda - m_{{j_{k-1}}}^{*, a_{k-1}}\big) g_{j_{k-1}}^{a_{k-1}-1, a_{k-1}}(S^k)
x_{p {j_{k-1}}}^{a_{k-1}}(S^{k-1}),
\end{align*}
where we have used parts (a, c) and Lemma \ref{lma:indextwoway}.

(f) The result follows from 
\begin{align*}
V_p(\lambda, S^k) & = F_p(S^k) + \lambda G_p(S^k) \\
& = F_p(S^{k+1}) - f_{{j_{k}}}^{a_{k}-1, a_{k}}(S^{k+1}) x_{p {j_{k}}}^{a_{k}}(S^{k}) \\
& \qquad + \lambda \big[G_p(S^{k+1}) + g_{j_{k}}^{a_{k}-1, a_{k}}(S^{k+1}) x_{p {j_{k}}}^{a_{k}}(S^{k})\big] \\
& = V_p(\lambda, S^{k+1}) - \big(f_{{j_{k}}}^{a_{k}-1, a_{k}}(S^{k+1}) - \lambda g_{{j_{k}}}^{a_{k}-1, a_{k}}(S^{k+1})\big)
x_{p {j_{k}}}^{a_{k}}(S^{k}) \\
& = V_p(\lambda, S^{k+1}) - \big(m_{{j_{k}}}^{a_{k}-1, a_{k}}(S^{k+1}) - \lambda\big) g_{{j_{k}}}^{a_{k}-1, a_{k}}(S^{k+1})
x_{p {j_{k}}}^{a_{k}}(S^{k}) \\
& = V_p(\lambda, S^{k-1}) - \big(\lambda - m_{{j_{k-1}}}^{*, a_{k-1}}\big) g_{j_{k-1}}^{a_{k-1}-1, a_{k-1}}(S^k)
x_{p {j_{k-1}}}^{a_{k-1}}(S^{k-1}),
\end{align*}
where we have used parts (b, d) and Lemma \ref{lma:indextwoway}.
\end{proof}

\section{Partial Conservation Laws}
\label{s:pcl}
This section shows that, under condition (PCLI1) in Definition \ref{def:pclind}, project performance metrics satisfy certain \emph{partial conservation laws} (PCLs), which extend those previously introduced by the author for finite-state restless (two-gear) bandits in~\cite{nmaap01,nmmp02}. It further uses those PCLs to lay further groundwork towards the proof of Theorem \ref{the:verthe}.

In the following result, we assume that the initial-state distribution $p$ has full support, which we write as $p > 0$.

\begin{proposition}[PCLs]
\label{pro:pcls} Suppose that  \textup{(PCLI1)} holds and let $p > 0$. Then, metrics $G_p(\pi)$ and $x_{pj}^{a}(\pi)$, for states $j$ and actions $a < A$, satisfy the following: for any admissible policy $\pi$ and $S \in \mathcal{F},$ 
\begin{itemize}
\item[(a.1)] $G_p(\pi)  + \sum_{a < a'} \sum_{j \in S_{a'}} g_{j}^{a, a'}(S) x_{pj}^{a}(\pi)  \geqslant 
G_p(S),$
with equality (conservation law),
\begin{equation}
\label{eq:csa1}
G_p(\pi)  + \sum_{a < a'} \sum_{j \in S_{a'}} g_{j}^{a, a'}(S) x_{pj}^{a}(\pi)  = 
G_p(S),
\end{equation}
 iff $\pi$ selects gears $a' \leqslant a$  in states $j \in S_{a}$ (so $\pi \preceq S$), for $a = 0, \ldots, A-1$.
 \item[(a.2)] In particular, for $S =  S^{K+1} = (\mathcal{N}, \emptyset, \ldots, \emptyset)$, it holds that
$G_p(\pi)    \geqslant G_p{(S^{K+1})}$, 
with equality  iff $\pi$ selects gear $0$ in every state.
\item[(a.3)] In the case $S =  S^1 = (\emptyset, \ldots, \emptyset, \mathcal{N})$, we have the conservation law
\begin{equation}
\label{eq:csa3}
G_p(\pi)  + \sum_{a < A} \sum_{j \in \mathcal{N}} g_{j}^{a, A}(S^1) x_{pj}^{a}(\pi)  = 
G_p{(S^1)}.
\end{equation}
\item[(b)] $\sum_{a < a'} \sum_{j \in S_{a'}} g_{j}^{a, a'}(S) x_{pj}^{a}(\pi)  \geqslant 
0$, 
with equality iff $\pi$ selects gears $a \geqslant a'$ in states $j \in S_{a'}$ (so $S \preceq \pi$), for $a' = 1, \ldots, A$.
\end{itemize}
\end{proposition}
\begin{proof}
(a.1) From Lemma \ref{lma:declaws} (b), we obtain, under condition (PCLI1),
\begin{equation}
\label{eq:a1ineq}
G_p(\pi)  + \sum_{a < a'} \sum_{j \in S_{a'}} g_{j}^{a, a'}(S) x_{pj}^a(\pi)  = 
G_p(S) + \sum_{a < a'} \sum_{j \in S_{a}} g_{j}^{a, a'}(S) x_{pj}^{a'}(\pi) \geqslant G_p(S),
\end{equation}
\textls[-15]{with equality iff (using the fact that the initial state distribution $p$ has full support) \mbox{$x_{pj}^{a'}(\pi) = 0$}} for $j \in S_{a}$ with $a < a'$, i.e., iff $\pi$ selects gears $a' \leqslant a$ in states $j \in S_{a}$, for $a = 0, \ldots, A-1$.

Parts (a.2) and (a.3) are direct consequences of (a.1).

Part (b) follows directly from (PCLI1).
\end{proof}

The next result shows how to reformulate the equivalent modified holding cost metric $\widehat{F}_p(\pi) \triangleq \sum_{a =0}^{A-1} \sum_{j \in \mathcal{N}} \widehat{h}_j^{a} x_{pj}^{a}(\pi)$ (see Section \ref{s:lpreflpp}) in terms of the output $\{(j_k, a_k), S^k, m_{j_k}^{*, a_k}\}_{k=1}^{K}$ of algorithm $\mathrm{DS}(\mathcal{F})$ in Algorithm \ref{alg:tdabutd} and expressions arising in the PCLs in Proposition \ref{pro:pcls}.

\begin{lemma}
\label{lma:objdual2}  Under condition (PCLI1), 
\[
\widehat{F}_p(\pi)  = 
m_{j_1}^{*,  a_1} \sum_{a < a'} \sum_{j \in S_{a'}^1} g_{j}^{a, a'}(S^1) x_{pj}^{a}(\pi)  + 
\sum_{k=2}^K (m_{j_k}^{*,  a_k} - m_{j_{k-1}}^{*, a_{k-1}}) \sum_{a < a'} \sum_{j \in S_{a'}^k} g_{j}^{a, a'}(S^k) x_{pj}^{a}(\pi).
\]
\end{lemma}
\begin{proof}
The result follows  from
Lemma \ref{lma:objdual}, which yields
\begin{align*}
\widehat{F}_p(\pi) & \triangleq \sum_{a =0}^{A-1} \sum_{j \in \mathcal{N}} \widehat{h}_j^{a} x_{pj}^{a}(\pi) = 
\sum_{l=1}^{K} \widehat{h}_{j_l}^{a_l-1} x_{p j_l}^{a_l-1}(\pi) \\
& = g_{j_1}^{a_1-1, A}(S^1) m_{j_1}^{*,  a_1} x_{p j_1}^{a_1-1}(\pi) \\
& \qquad + 
  \sum_{l=2}^{K} \bigg[g_{j_l}^{a_l-1, A}(S^1) m_{j_1}^{*,  a_1}  + \sum_{k=2}^{l} g_{j_l}^{a_l-1, a^k(j_l)}(S^k) (m_{j_k}^{*, a_k}-m_{j_{k-1}}^{*, a_{k-1}})\bigg] x_{p j_l}^{a_l-1}(\pi)  \\
& = m_{j_1}^{*,  a_1} \sum_{l=1}^{K} g_{j_l}^{a_l-1, A}(S^1) x_{p j_l}^{a_l-1}(\pi) + 
\sum_{k=2}^{K} (m_{j_k}^{*, a_k}-m_{j_{k-1}}^{*, a_{k-1}}) \sum_{l=k}^{K} g_{j_l}^{a_l-1, a^k(j_l)}(S^k) x_{p j_l}^{a_l-1}(\pi) \\
& = m_{j_1}^{*,  a_1} \sum_{a < a'} \sum_{j \in S_{a'}^1} g_{j}^{a, a'}(S^1) x_{pj}^{a}(\pi)  + 
\sum_{k=2}^K (m_{j_k}^{*,  a_k} - m_{j_{k-1}}^{*, a_{k-1}}) \sum_{a < a'} \sum_{j \in S_{a'}^k} g_{j}^{a, a'}(S^k) x_{pj}^{a}(\pi).
\end{align*}
\end{proof}

The following result draws on the previous one by showing how to reformulate the cost metric $\widehat{V}_p(\lambda, \pi)$ of the equivalent modified $\lambda$-price problem (see Lemma \ref{lma:equivPlhatPl}) in terms of the output of algorithm $\mathrm{DS}(\mathcal{F})$.
\begin{lemma}
\label{lma:objduallamb}  \textls[-15]{Under (PCLI1), $\widehat{V}_p(\lambda, \pi)$ can be reformulated into the following equivalent expressions:}
\begin{itemize}
\item[(a)] 
\begin{align*}
\widehat{V}_p(\lambda, \pi) & = \lambda G_p(S^1) + (m_{j_1}^{*,  a_1} - \lambda) \sum_{a < a'} \sum_{j \in S_{a'}^1} g_{j}^{a, a'}(S^1) x_{pj}^{a}(\pi) \\
& \qquad  + 
\sum_{k=2}^K (m_{j_k}^{*,  a_k} - m_{j_{k-1}}^{*, a_{k-1}}) \sum_{a < a'} \sum_{j \in S_{a'}^k} g_{j}^{a, a'}(S^k) x_{pj}^{a}(\pi);
\end{align*}
\item[(b)] 
\begin{align*}
\widehat{V}_p(\lambda, \pi) & = m_{j_1}^{*,  a_1} G_p(S^1)   + 
\sum_{k=2}^{l-1} (m_{j_k}^{*,  a_k} - m_{j_{k-1}}^{*, a_{k-1}}) \bigg(G_p(\pi)  + \sum_{a < a'} \sum_{j \in S_{a'}^k} g_{j}^{a, a'}(S^k) x_{pj}^{a}(\pi)\bigg) \\
& \qquad + (\lambda -  m_{j_{l-1}}^{*,  a_{l-1}}) \bigg(G_p(\pi)  + \sum_{a < a'} \sum_{j \in S_{a'}^l} g_{j}^{a, a'}(S^l) x_{pj}^{a}(\pi)\bigg) \\
& \qquad + (m_{j_{l}}^{*, a_{l}} -  \lambda)  \sum_{a < a'} \sum_{j \in S_{a'}^l} g_{j}^{a, a'}(S^l) x_{pj}^{a}(\pi) \\
& \qquad  + 
\sum_{k=l+1}^{K} (m_{j_k}^{*,  a_k} - m_{j_{k-1}}^{*, a_{k-1}}) \sum_{a < a'} \sum_{j \in S_{a'}^k} g_{j}^{a, a'}(S^k) x_{pj}^{a}(\pi);
\end{align*}
\item[(c)] 
\begin{align*}
\widehat{V}_p(\lambda, \pi) & = m_{j_1}^{*,  a_1} G_p(S^1)   + 
\sum_{k=2}^K (m_{j_k}^{*,  a_k} - m_{j_{k-1}}^{*, a_{k-1}}) \bigg(G_p(\pi)  + \sum_{a < a'} \sum_{j \in S_{a'}^k} g_{j}^{a, a'}(S^k) x_{pj}^{a}(\pi)\bigg) \\
& \qquad + ( \lambda -  m_{j_K, a_K}) G_p(\pi).
\end{align*}
\end{itemize}
\end{lemma}
\begin{proof}
All three parts follow from the definition $\widehat{V}_p(\lambda, \pi) \triangleq \widehat{F}_p(\pi) + \lambda G_p(\pi)$ and \mbox{Lemma \ref{lma:objdual2}} by suitably rearranging terms. 
\end{proof}

The next result draws on the above to reformulate the cost metrics $\widehat{V}_p(\lambda, S^l)$  in terms of the output of algorithm $\mathrm{DS}(\mathcal{F})$.

\begin{lemma}
\label{lma:objduallamb2}  Under condition (PCLI1),
\begin{itemize}
\item[(a)] 
$
\widehat{V}_p(\lambda, S^1) = \lambda G_p(S^1);
$
\item[(b)] For $2 \leqslant l \leqslant K$,
\[
\widehat{V}_p(\lambda, S^l) = m_{j_1}^{*,  a_1} G_p(S^1)   + 
\sum_{k=2}^{l-1} (m_{j_k}^{*,  a_k} - m_{j_{k-1}}^{*, a_{k-1}}) G_p(S^k) + (\lambda -  m_{j_{l-1}}^{*,  a_{l-1}}) G_p(S^l);
\]
\item[(c)] 
$
\widehat{V}_p(\lambda, S^{K+1}) = m_{j_1}^{*,  a_1} G_p(S^1)   + 
\sum_{k=2}^K (m_{j_k}^{*,  a_k} - m_{j_{k-1}}^{*, a_{k-1}}) G_p(S^k)  + ( \lambda -  m_{j_K}^{*, a_K}) G_p(S^{K+1}).
$
\end{itemize}
\end{lemma}
\begin{proof}
Each part follows from the corresponding part in Lemma \ref{lma:objduallamb}, using the PCLs in Proposition \ref{pro:pcls} to simplify the resulting expressions. 
\end{proof}

\section{Proof of the Verification Theorem}
\label{s:proofotvthe}
We are now ready to prove Theorem \ref{the:verthe}.

\begin{proof}[Proof of Theorem \ref{the:verthe}]
We will prove the result by showing the following: (i) policy $S^1$ is $\lambda$-optimal iff $\lambda \leqslant m_{j_1}^{*,  a_1}$;
(ii) for $2 \leqslant l \leqslant K$, policy $S^l$ is $\lambda$-optimal iff $m_{j_{l-1}}^{*,  a_{l-1}} \leqslant \lambda \leqslant m_{j_l}^{*,  a_l}$; and 
(iii) policy $S^{K+1}$ is $\lambda$-optimal iff $\lambda \geqslant m_{j_K}^{*,  a_K}$.
Note that (i, ii, iii) imply that the model is $\mathcal{F}$-indexable with DAI $\lambda_{j}^{*, a}$ being given by the MPI $m_{j}^{*, a}$.

We consider below that $p > 0$, i.e., the initial-state distribution $p$ has full support.

Start with (i).  If $\lambda \leqslant m_{j_1}^{*,  a_1}$, we have, for any policy $\pi$, 
\begin{align*}
\widehat{V}_p(\lambda, \pi) & = \lambda G_p(S^1) + (m_{j_1}^{*,  a_1} - \lambda) \sum_{a < a'} \sum_{j \in S_{a'}^1} g_{j}^{a, a'}(S^1) x_{pj}^{a}(\pi) \\
& \qquad  + 
\sum_{k=2}^K (m_{j_k}^{*,  a_k} - m_{j_{k-1}}^{*, a_{k-1}}) \sum_{a < a'} \sum_{j \in S_{a'}^k} g_{j}^{a, a'}(S^k) x_{pj}^{a}(\pi) \\
& \geqslant \lambda G_p(S^1) = \widehat{V}_p(\lambda, S^1),
\end{align*}
where we have used Lemmas \ref{lma:objduallamb} (a) and \ref{lma:objduallamb2} (a) and conditions (PCLI1, PCLI2).
Hence, policy $S^1$ is $\lambda$-optimal.

Conversely, suppose that policy $S^1$ is $\lambda$-optimal. Then, using Lemma \ref{lma:lambdapprec} (f), we obtain
\[
0 \leqslant V_p(\lambda, S^{2}) - V_p(\lambda, S^1) = (m_{j_{1}}^{*, a_{1}} - \lambda) 
  g_{j_{1}}^{a_{1}-1, a_{1}}(S^{2}) x_{p j_{1}}^{a_{1}}(S^{1}).
\]

Now, since $g_{j_{1}}^{a_{1}-1, a_{1}}(S^{2}) > 0$ (by (PCLI1)) and $x_{p j_{1}}^{a_{1}}(S^{1}) > 0$ (because $p$ has full support), it follows that $\lambda \leqslant m_{j_{1}}^{*, a_{1}}$.

Consider now (ii). If $m_{j_{l-1}}^{*,  a_{l-1}} \leqslant \lambda \leqslant m_{j_l}^{*,  a_l}$ for some $l$ with $2 \leqslant l \leqslant K$, we have, for any policy $\pi$, 
\begin{align*}
\widehat{V}_p(\lambda, \pi) & = m_{j_1}^{*,  a_1} G_p(S^1)   + 
\sum_{k=2}^{l-1} (m_{j_k}^{*,  a_k} - m_{j_{k-1}}^{*, a_{k-1}}) \bigg(G_p(\pi)  + \sum_{a < a'} \sum_{j \in S_{a'}^k} g_{j}^{a, a'}(S^k) x_{pj}^{a}(\pi)\bigg) \\
& \qquad + (\lambda -  m_{j_{l-1}}^{*,  a_{l-1}}) \bigg(G_p(\pi)  + \sum_{a < a'} \sum_{j \in S_{a'}^l} g_{j}^{a, a'}(S^l) x_{pj}^{a}(\pi)\bigg) \\
& \qquad + (m_{j_{l}}^{*, a_{l}} -  \lambda)  \sum_{a < a'} \sum_{j \in S_{a'}^l} g_{j}^{a, a'}(S^l) x_{pj}^{a}(\pi) \\
& \qquad  + 
\sum_{k=l+1}^{K} (m_{j_k}^{*,  a_k} - m_{j_{k-1}}^{*, a_{k-1}}) \sum_{a < a'} \sum_{j \in S_{a'}^k} g_{j}^{a, a'}(S^k) x_{pj}^{a}(\pi) \\
& \geqslant \widehat{V}_p(\lambda, S^l),
\end{align*}
where we have further used Lemmas \ref{lma:objduallamb}(b) and \ref{lma:objduallamb2}(b), Proposition \ref{pro:pcls}, and conditions (PCLI1, PCLI2).
Hence, policy $S^l$ is $\lambda$-optimal.

Conversely, suppose that policy $S^l$ is $\lambda$-optimal. Then, using Lemma \ref{lma:lambdapprec} (e, f), we obtain
\[
0 \leqslant V_p(\lambda, S^{l+1}) - V_p(\lambda, S^l) = (m_{j_{l}}^{*, a_{l}} - \lambda) 
  g_{j_{l}}^{a_{l}-1, a_{l}}(S^{l+1}) x_{p j_{l}}^{a_{l}}(S^{l})
\]
and
\[
0 \leqslant V_p(\lambda, S^{l-1}) - V_p(\lambda, S^l) = (\lambda - m_{j_{l-1}}^{*, a_{l-1}}) 
  g_{j_{l-1}}^{a_{l-1}-1, a_{l-1}}(S^l) x_{p j_{l-1}}^{a_{l-1}}(S^{l-1}).
\]

Now, since $g_{j_{l}}^{a_{l}-1, a_{l}}(S^{l+1}) > 0$, $g_{j_{l-1}}^{a_{l-1}-1, a_{l-1}}(S^l) > 0$, $x_{p j_{l}}^{a_{l}}(S^{l}) > 0$ and $x_{p j_{l-1}}^{a_{l-1}}(S^{l-1}) > 0$,  it follows that $m_{j_{l-1}}^{*,  a_{l-1}} \leqslant \lambda \leqslant m_{j_l}^{*,  a_l}$.

Finally, consider (iii). 
If $\lambda \geqslant m_{j_K}^{*,  a_K}$, we can write, for any policy $\pi$, 
\begin{align*}
\widehat{V}_p(\lambda, \pi) & = m_{j_1}^{*,  a_1} G_p(S^1)   + 
\sum_{k=2}^K (m_{j_k}^{*,  a_k} - m_{j_{k-1}}^{*, a_{k-1}}) \bigg(G_p(\pi)  + \sum_{a < a'} \sum_{j \in S_{a'}^k} g_{j}^{a, a'}(S^k) x_{pj}^{a}(\pi)\bigg) \\
& \qquad + ( \lambda -  m_{j_K, a_K}) G_p(\pi) \\
& \geqslant \widehat{V}_p(\lambda, S^{K+1}), 
\end{align*}
where we have further used Lemmas \ref{lma:objduallamb} (c) and \ref{lma:objduallamb2} (c), Proposition \ref{pro:pcls}, and conditions (PCLI1, PCLI2).
Hence, policy $S^{K+1}$ is $\lambda$-optimal.

Conversely, suppose that policy $S^{K+1}$ is $\lambda$-optimal. Then, using Lemma \ref{lma:lambdapprec} (f), we obtain
\[
0 \leqslant V_p(\lambda, S^{K}) - V_p(\lambda, S^{K+1}l) = (\lambda - m_{j_{K}}^{*, a_{K}}) 
  g_{j_{K}}^{a_{K}-1, a_{K}}(S^{K+1}) x_{p j_{K}}^{a_{K}}(S^{K}).
\]

Now, since $g_{j_{K}}^{a_{K}-1, a_{K}}(S^{K+1}) > 0$ and $x_{p j_{K}}^{a_{K}}(S^{K}) > 0$,  it follows that $\lambda \geqslant m_{j_{K}}^{*, a_{K}}$. This completes the proof.
\end{proof}

\section{Application to Multi-Armed Multi-Gear Bandit Problem: Bound and Index Policy}
\label{s:bpMAMGBP}
\subsection{The Multi-armed Multi-Gear Bandit Problem (MAMGBP)}
\label{s:MAMGBP}
Besides the intrinsic interest of the indexability property in Definition \ref{def:indxb} for solving optimally the multi-gear bandit model, we next discuss as further motivation for such a property its application to design a suboptimal heuristic policy for the intractable \emph{multi-armed multi-gear bandit problem} (MAMGBP) introduced by the author in~\cite{nmvaluet08a} (where it was called the \emph{multi-armed multi-mode bandit problem}).

The MAMGBP concerns the optimal dynamic allocation of a single shared resource to a finite collection of $L$  projects modeled as multi-gear bandits, subject to a \emph{peak resource constraint} stating that the total resource usage in each period cannot exceed a given amount $\bar{q}$.
Denote by $s_l(t)$ and $a_l(t)$ the state and the action at time $t$ for project $l = 1, \ldots, L$, which belong to the state and action spaces $\mathcal{N}_l = \{1, \ldots, N_l\}$ and $\mathcal{A}_l = \{0, \ldots, A_l\}$, respectively.
The parameters of project $l$ are denoted here by $h_{l}(j_l, a_l)$, $q_{l}(j_l, a_l)$, and $p_{l}^a(i_l, j_l)$.

The MAMGBP is a multi-dimensional MDP with \emph{joint state} $\mathbf{s}(t) = (s_l(t))_{l=1}^L$ belonging to the \emph{joint state space} $\boldsymbol{\mathcal{N}} \triangleq \prod_{l=1}^L \mathcal{N}_l$ and \emph{joint action} $\mathbf{a}(t) = (a_l(t))_{l=1}^L$. 

The \emph{joint holding cost} and \emph{joint resource consumption} are additive across projects, being $h(\mathbf{i}, \mathbf{a}) \triangleq \sum_{l=1}^L h_l(i_l, a_l)$ and $q(\mathbf{i}, \mathbf{a}) \triangleq \sum_{l=1}^L q_l(i_l, a_l)$ in joint state $\mathbf{i} = (i_l)_{l=1}^L$ under joint action $\mathbf{a} = (a_l)_{l=1}^L$.
The set of \emph{feasible actions}  in joint state $\mathbf{i}$, satisfying the aforementioned peak resource constraint, is 
\begin{equation}
\label{eq:prct}
\boldsymbol{\mathcal{A}}(\mathbf{i}) \triangleq \bigg\{\mathbf{a} \in \prod_{l=1}^L \mathcal{A}_l \colon q(\mathbf{i}, \mathbf{a}) \leqslant \bar{q}\bigg\}.
\end{equation}

To ensure that there 
always exists a feasible joint action,  we require that, for every joint state $\mathbf{i}$, 
\begin{equation}
\label{eq:atleastofja}
q(\mathbf{i}, \mathbf{0}) \leqslant \bar{q}.
\end{equation}

Individual project state transitions are conditionally independent given that the actions at every project have been selected, and hence the
\emph{joint transition probabilities} are multiplicative across projects, being given by 
$p_l^{\mathbf{a}}(\mathbf{i}, \mathbf{j}) \triangleq \prod_{l=1}^L p_{l}^{a_l}(i_l, j_l)$.

Let $\boldsymbol{\Pi}(\bar{q})$ be the class of history-dependent randomized policies for selecting a feasible joint action at each time period, where we make explicit its dependence on $\bar{q}$ and denote by $\Ex_{\mathbf{i}}^{\boldsymbol{\pi}}[\cdot]$ the expectation under policy $\boldsymbol{\pi} \in \boldsymbol{\Pi}(\bar{q})$ starting from the joint state $\mathbf{i}$.   The expected total discounted holding cost incurred under policy $\boldsymbol{\pi}$ starting from $\mathbf{i}$ is 
\[
F(\mathbf{i}, \boldsymbol{\pi}) \triangleq \Ex_{\mathbf{i}}^{\boldsymbol{\pi}}\bigg[\sum_{l=1}^L \sum_{t=0}^\infty h_l(s_l(t), a_l(t)) \beta^t\bigg],
\]
and hence the optimal holding cost is
\[
F^*(\mathbf{i}) \triangleq \inf \, \{F(\mathbf{i}, \boldsymbol{\pi})\colon \boldsymbol{\pi} \in \boldsymbol{\Pi}(\bar{q})\}.
\]

We can thus formulate the MAMGBP as follows:
\begin{equation}
\label{eq:MAMGBP}
(P) \quad \textup{find } \boldsymbol{\pi}^* \in \boldsymbol{\Pi}(\bar{q})\colon F(\mathbf{i}, \boldsymbol{\pi}^*) = 
F^*(\mathbf{i}), \quad \mathbf{i} \in \boldsymbol{\mathcal{N}}.
\end{equation}

We shall refer to a policy $\boldsymbol{\pi}^*$ solving the MAMGBP $(P)$ as a \emph{$P$-optimal policy}.

Again, standard results in MDP theory ensure the existence of a $P$-optimal policy $\boldsymbol{\pi}^*$ in the class $\boldsymbol{\Pi}^{\textup{SD}}$ of stationary deterministic policies, which is determined by the Bellman equations
\begin{equation}
\label{eq:bemamdbp}
F^*(\mathbf{i}) = \min_{\mathbf{a} \in \boldsymbol{\mathcal{A}}(\mathbf{i})} \, h(\mathbf{i}, \mathbf{a}) +  \beta \sum_{\mathbf{j} \in \boldsymbol{\mathcal{N}}} p_l^{\mathbf{a}}(\mathbf{i}, \mathbf{j}) F^*(\mathbf{j}), \enspace 
\mathbf{i} \in \boldsymbol{\mathcal{N}}.
\end{equation}

However, these equations are hindered by the \emph{curse of dimensionality}, as the size of the state space $\boldsymbol{\mathcal{N}}$ grows exponentially with the number $L$ of projects, which renders them computationally intractable in practice for all but small $L$.

\subsection{A Bound for the MAMGBP}
\label{s:bMAMGBP}
Ref.~\cite{nmvaluet08a} introduced a Lagrangian approach for obtaining a lower bound for the optimal value of the  MAMGBP, extending that of Whittle~\cite{whit88b} for the case of two-gear projects.
First, we construct a \emph{relaxation} of problem $(P)$ by (i) relaxing the class of admissible policies from $\boldsymbol{\Pi}(\bar{q})$ to $\boldsymbol{\Pi}(\infty)$, thus allowing violations to 
 the sample-path peak resource constraints
\[
q(\mathbf{s}(t), \mathbf{a}(t)) \leqslant \bar{q}, \enspace t = 0, 1, \ldots,
\]
and (ii) replacing the latter
by the following aggregate relaxed version in expectation:
\[
\Ex_{\mathbf{i}}^{\boldsymbol{\pi}}\bigg[\sum_{t=0}^\infty q(\mathbf{s}(t), \mathbf{a}(t)) \beta^t\bigg] \leqslant \frac{\bar{q}}{1-\beta}.
\]

This leads to the following \emph{relaxation} of problem $(P)$ in (\ref{eq:MAMGBP}):
\begin{equation}
\label{eq:relxpP}
\begin{split}
(\widehat{P}) \quad & \minim \,  \Ex_{\mathbf{i}}^{\boldsymbol{\pi}}\bigg[\sum_{t=0}^\infty h(\mathbf{s}(t), \mathbf{a}(t)) \beta^t\bigg] \\
& \st\colon  \boldsymbol{\pi} \in \boldsymbol{\Pi}(\infty) \\
&  \Ex_{\mathbf{i}}^{\boldsymbol{\pi}}\bigg[\sum_{t=0}^\infty q(\mathbf{s}(t), \mathbf{a}(t)) \beta^t\bigg] \leqslant \frac{\bar{q}}{1-\beta}.
\end{split}
\end{equation}

The \emph{relaxed problem} $(\widehat{P})$ is a \emph{constrained MDP} (see, e.g.,~\cite{altman99}), for which an optimal policy generally depends on the initial state $\mathbf{i}$.
Such problems are amenable to a Lagrangian approach. Introducing a non-negative multiplier 
$\lambda \geqslant 0$ attached to the constraint in (\ref{eq:relxpP}), we can \emph{dualize} the latter, i.e., bring it into the objective, obtaining the \emph{Lagrangian relaxation}
\begin{equation}
\label{eq:LRrelxpP}
\begin{split}
(\widehat{P}_{\lambda}) \quad & \minim_{\boldsymbol{\pi} \in \boldsymbol{\Pi}(\infty)} \,  \Ex_{\mathbf{i}}^{\boldsymbol{\pi}}\bigg[\sum_{t=0}^\infty h(\mathbf{s}(t), \mathbf{a}(t)) \beta^t\bigg] + \lambda \bigg(\Ex_{\mathbf{i}}^{\boldsymbol{\pi}}\bigg[\sum_{t=0}^\infty q(\mathbf{s}(t), \mathbf{a}(t)) \beta^t\bigg] - \frac{\bar{q}}{1-\beta} \bigg).
\end{split}
\end{equation} 

Note that, for any initial joint state $\mathbf{i}$ and multiplier $\lambda \geqslant 0$, the optimal cost $\widehat{V}^*(\mathbf{i}, \lambda)$ of $(\widehat{P}_\lambda)$ is a \emph{lower bound} for that of relaxed problem $(\widehat{P})$, which we denote by $\widehat{F}^*(\mathbf{i})$. In turn, the latter gives a lower bound for the optimal cost $F^*(\mathbf{i})$ of $(P)$. Thus, 
\begin{equation}
\label{eq:lowbnds}
\widehat{V}^*(\mathbf{i}, \lambda) \leqslant \widehat{F}^*(\mathbf{i}) \leqslant F^*(\mathbf{i}).
\end{equation}

In light of (\ref{eq:lowbnds}), we are interested in finding an optimal multiplier $\lambda^*(\mathbf{i})$ solving the \emph{dual problem}
\begin{equation}
\label{eq:dualP}
(D) \quad \maxim_{\lambda \geqslant 0} \,  \widehat{V}^*(\mathbf{i}, \lambda).
\end{equation} 

Since $\widehat{V}^*(\mathbf{i}, \lambda)$ is a \emph{concave function} of $\lambda$, being a minimum of linear functions of $\lambda$, a local maximum of problem $(D)$ will be a global maximum.
Furthermore, since the above problems can be formulated as finite \emph{linear optimization} (LO) problems, the \emph{strong duality} property of the latter ensures the existence of an optimal multiplier $\lambda^*(\mathbf{i}) \geqslant 0$ solving  problem $(D)$ which attains the upper bound $\widehat{F}^*(\mathbf{i})$, i.e., with $\widehat{V}^*(\mathbf{i}, \lambda^*(\mathbf{i})) = \widehat{F}^*(\mathbf{i})$. This corresponds to the satisfaction of the \emph{complementary slackness} property
\begin{equation}
\label{eq:complSl}
\lambda^*(\mathbf{i}) \bigg(\Ex_{\mathbf{i}}^{\boldsymbol{\pi}^*}\bigg[\sum_{t=0}^\infty q(\mathbf{s}(t), \mathbf{a}(t)) \beta^t\bigg] - \frac{\bar{q}}{1-\beta} \bigg) = 0,
\end{equation} 
where $\boldsymbol{\pi}^*$ is an optimal policy for problem $(\widehat{P}_{\lambda^*(\mathbf{i})})$.

Now, since individual project state transitions are conditionally independent given that a joint action has been selected, it suffices ---as noted by Whittle~\cite{whit88b} for the case of two-gear projects--- to consider in $(\widehat{P}_{\lambda})$ \emph{decoupled policies} $\boldsymbol{\pi} = (\pi_l)_{l=1}^L$, where $\pi_l \in \Pi_l$ and $\Pi_l$ is the class of admissible policies for operating  project $l$ \emph{as if it were in isolation}.  This allows us to reformulate problem $(\widehat{P}_{\lambda})$ as 
\begin{equation}
\label{eq:dLRrelxpP}
\begin{split}
(\widehat{P}_{\lambda}) \quad & \minim \,  \sum_{l=1}^L \Ex_{i_l}^{\pi_l}\bigg[ \sum_{t=0}^\infty \big(h_l(s_l(t), a_l(t)) + \lambda q_l(s_l(t), a_l(t))\big) \beta^t\bigg] - \lambda \frac{\bar{q}}{1-\beta} \\
& \st\colon  \pi_l \in \Pi_l, \enspace l = 1, \ldots, L.
\end{split}
\end{equation} 

We can thus \emph{decouple} problem $(\widehat{P}_{\lambda})$ into the \emph{individual project subproblems}
\begin{equation}
\label{eq:ipsdLRrelxpP}
(\widehat{P}_{l, \lambda}) \quad \minim_{\pi_l \in \Pi_l} \,  \Ex_{i_l}^{\pi_l}\bigg[ \sum_{t=0}^\infty \big(h_l(s_l(t), a_l(t)) + \lambda q_l(s_l(t), a_l(t))\big) \beta^t\bigg],
\end{equation} 
for $l = 1, \ldots, L$.
Denoting by $V_l^*(i_l, \lambda)$ the minimum cost objective of subproblem $(\widehat{P}_{l, \lambda})$, it follows that the optimal cost $\widehat{V}^*(\mathbf{i}, \lambda)$ of Lagrangian relaxation $(\widehat{P}_\lambda)$ is decoupled as 
\begin{equation}
\label{eq:hatVstarilam}
\widehat{V}^*(\mathbf{i}, \lambda) = \sum_{l=1}^L V_l^*(i_l, \lambda) - \lambda \frac{\bar{q}}{1-\beta},
\end{equation}
which allows us to reformulate dual problem $(D)$ in (\ref{eq:dualP}) as
\begin{equation}
\label{eq:dualP2}
(D) \quad \maxim_{\lambda \geqslant 0} \,  \sum_{l=1}^L V_l^*(i_l, \lambda) - \lambda \frac{\bar{q}}{1-\beta}
\end{equation}

Now, suppose that each project $l$ is indexable with DAI $\lambda_l^*(j_l, a_l)$, so 
 such indices characterize as in Definition \ref{def:indxb}  the optimal policies for  individual project subproblems  $(\widehat{P}_{l, \lambda})$, which facilitates the evaluation of optimal costs $V_l^*(i_l, \lambda)$ and hence the computational solution of dual problem $(D)$. 
For such a purpose, one can use the result that, if $\pi_l^*(\lambda)$ is an optimal policy for subproblem $(\widehat{P}_{l, \lambda})$, then $-\Ex_{i_l}^{\pi_l^*(\lambda)}\big[\sum_{t=0}^\infty q_l(s_l(t), a_l(t)) \beta^t\big]$ is a \emph{subgradient} of 
$V_l^*(i_l, \lambda)$, seen as a function of $\lambda$.

\subsection{A Downshift Index Policy for the MAMGBP}
\label{s:ipMAMGBP}
Assuming that individual projects are indexable, the author proposed in~\cite{nmvaluet08a} a suboptimal heuristic index policy for the above MAMGBP based on the projects' DAIs.
Here we present a different proposal of a heuristic index policy based on individual project DAIs, which is more easily implementable than that in~\cite{nmvaluet08a}. 

Suppose that at time $t$ the joint state is $\mathbf{j} = (j_l)_{l=1}^L$. 
Consider the project DAIs evaluated at such states, $\lambda_l^*(j_l, a_l)$, for project $l = 1, \ldots, L$.
The proposed index policy is described in Algorithm \ref{alg:indexPolMAMGBP}, which specifies how to obtain the 
joint action $\widehat{\mathbf{a}} = (\widehat{a}_l)_{l=1}^L$ prescribed in such a joint state.

In short, the algorithm starts by assigning the highest possible gear $A_l$ to each project $l$. If this is feasible, in that it does not violate the peak resource constraint, this would be the prescribed joint action.
Otherwise, the algorithm proceeds by downshifting one of the projects to the next lower gear. The chosen project is one that has minimum DAI at the current gear.
The algorithm proceeds until the peak resource constraint is satisfied and the DAIs at the projects with prescribed active actions, if any, are non-negative. In light of the above, we call the policy resulting from this algorithm the  \emph{downshift index policy}.

The intuition behind the design of such a policy is that projects should be operated in such a way that two conflicting goals are balanced: (1) higher gears are to be preferred to lower gears whenever possible; and (2) the resulting joint action must be feasible, satisfying the peak resource constraint. 
The proposed downshift index policy is designed to strike such a balance. If a joint action is not feasible so that a project must be downshifted to a lower gear, the chosen project is one where the loss in performance due to such a downshift, for which the project DAIs are used as a proxy measure, is minimal.

Note that the downshift index policy reduces to the Whittle index policy in the case of two-gear projects.

\vspace{12pt}
\begin{algorithm}[H]
\caption{\hl{Downshift} 
 index policy for the MAMGBP.}
\begin{minipage}{3in}
\textbf{Input:} $\mathbf{j} = (j_l)_{l=1}^L$  (current joint state) \\
\textbf{Output:} $\widehat{\mathbf{a}} = (\widehat{a}_l)_{l=1}^L$  (prescribed joint action)
\begin{tabbing}
\textit{Initialization:} 
$a_l := A_l$, \enspace $l = 1, \ldots, L$  \\
\textit{Loop:} \\
\textbf{while} \= \enspace $\sum_{l=1}^L q_l(j_l, a_l) > \bar{q}$ \enspace \textbf{or} \enspace $\min_{1 \leq l \leq L\colon a_l \geqslant 1} \, \lambda_l^*(j_l, a_l) \leqslant 0$ \enspace \textbf{do} \\
\> \textbf{pick}  
 $\hat{l} \in \argmin_{1 \leq l \leq L\colon a_l \geqslant 1} \, \lambda_l^*(j_l, a_l)$ \\
\>   $a_{\hat{l}} := a_{\hat{l}} - 1$ (downshift gear in project $\hat{l}$) \\
\textbf{end} \{ while \} \\
$\widehat{\mathbf{a}} := \mathbf{a} = (j_l)_{l=1}^L$
\end{tabbing}
\end{minipage}
\label{alg:indexPolMAMGBP}
\end{algorithm}

\section{Some Extensions}
\label{s:sext}
This section presents some extensions to the above framework.
\subsection{Extension to the Long-Run Average Cost Criterion}
\label{s:elracc}
The above results for the discounted cost criterion readily extend to the \emph{(long-run) average cost criterion} (see, e.g., (\cite{put94} Ch.\ 8)) under appropriate \emph{ergodicity} conditions.
Consider the average cost, including holding and resource usage costs (charged at price $\lambda$),  of running the project starting from state $i$ under a policy $\pi \in \Pi$, defined by
\[
\overline{V}_i(\lambda, \pi) \triangleq \limsup_{T \to \infty} \, \frac{1}{T} \Ex_i^{\pi}\bigg[\sum_{t=0}^{T-1} \big(h_{s(t)}^{a(t)} + \lambda q_{s(t)}^{a(t)}\big)\bigg].
\]

We further define the corresponding optimal cost
\[
\overline{V}_i^*(\lambda) \triangleq \inf_{\pi \in \Pi} \, \overline{V}_i(\lambda, \pi).
\]

We can thus formulate the project's \emph{average $\lambda$-price problem} as
\begin{equation}
\label{eq:aclambdapp}
(\overline{P}_\lambda) \quad \textup{find } \pi^*(\lambda) \in \Pi\colon \overline{V}_i(\lambda, \pi^*(\lambda)) = 
\overline{V}_i^*(\lambda), \quad i \in \mathcal{N}.
\end{equation}

We shall refer to a policy $\pi^*(\lambda)$ solving the average $\lambda$-price problem $(\overline{P}_\lambda)$ as a 
\emph{$\lambda$-optimal policy}.

We shall make the following assumption.
\begin{assumption}
\label{ass:avccrit} {The following conditions hold:
\begin{itemize}
\item[\textup{(i)}] The model is {weakly accessible}, so the state space $\mathcal{N}$ can be partitioned into two subsets $\mathcal{N}^{\textup{tr}}$ and  $\mathcal{N}^{\textup{acc}}$, such that (i.a) all states in $\mathcal{N}^{\textup{tr}}$ are transient under every stationary policy and (i.b) for every two states $i$ and $j$ in  $\mathcal{N}^{\textup{acc}}$, $j$ is {accessible from} $i$, so there exists a stationary policy $\pi$ and a positive integer $t$ such that $\mathbb{P}_i^\pi\{s(t) = j\} > 0$. 
\item[\textup{(ii)}] Every policy $S \in \mathcal{F}$ is {unichain}, i.e., it induces a single recurrent class plus  possible additional transient states.
\end{itemize}}
\end{assumption}

Now, by standard results in average-cost MDP theory (see \cite{bertsek12} (Sec.\ 5.2)), \mbox{Assumption \ref{ass:avccrit}} (i) ensures the existence of a 
$\lambda$-optimal policy $\pi^*(\lambda) \in \Pi^{\textup{SD}}$, with the optimal average cost $\overline{V}_i^*(\lambda)$ being independent of the initial state, i.e., $\overline{V}_i^*(\lambda) \equiv \overline{V}^*(\lambda)$.

Now, by using the \emph{Laurent series expansions} for finite-state and -action MDP models (see Corollaries 8.2.4 and 8.2.5 in~\cite{put94}), we have the following.
For any stationary deterministic policy $\pi \in \Pi^{\textup{SD}}$,
\[
\overline{x}_{ij}^a(\pi) \triangleq \lim_{T \to \infty} \, \frac{1}{T} \Ex_i^\pi\bigg[\sum_{t=0}^{T-1} 1_{\{a(t) = a\}}\bigg] = \lim_{\beta \nearrow 1} \, (1-\beta) x_{ij}^a(\pi),
\] 
\[
\overline{F}_{i}(\pi) \triangleq \lim_{T \to \infty} \, \frac{1}{T} \Ex_i^\pi\bigg[\sum_{t=0}^{T-1} h_{s(t)}^{a(t)}\bigg] = \lim_{\beta \nearrow 1} \, (1-\beta) F_{i}(\pi),
\] 
\[
\overline{G}_{i}(\pi) \triangleq \lim_{T \to \infty} \, \frac{1}{T} \Ex_i^\pi\bigg[\sum_{t=0}^{T-1} q_{s(t)}^{a(t)}\bigg] = \lim_{\beta \nearrow 1} \, (1-\beta) G_{i}(\pi).
\] 

Furthermore, for any $S = (S_0, \ldots, S_A) \in \mathcal{F}$, Assumption \ref{ass:avccrit}(ii) ensures that the above metrics do not depend on the initial state $i$, so we can write $\overline{x}_{j}^a(S)$, $\overline{F}(S)$, and $\overline{G}(S)$.
Furthermore, we have the Laurent series expansions
\[
F_{i}(S) = \frac{\overline{F}(S)}{1-\beta} + \varphi_{i}(S) + O(1-\beta), \enspace \textup{as}  \enspace \beta \nearrow 1
\]
and 
\[
G_{i}(S) = \frac{\overline{G}(S)}{1-\beta} + \gamma_{i}(S) + O(1-\beta), \enspace \textup{as}  \enspace \beta \nearrow 1,
\]
where the \emph{bias} terms $\varphi_{i}(S)$ and $\gamma_{i}(S)$ are determined, up to an additive constant, by the evaluation equations
\[
\overline{F}(S) + \varphi_{i}(S) = h_i^a + \sum_{j \in \mathcal{N}} p_{ij}^a \varphi_{j}(S), \enspace i \in S_a, a \in \mathcal{A}
\]
and
\[
\overline{G}(S) + \gamma_{i}(S) = q_i^a + \sum_{j \in \mathcal{N}} p_{ij}^a \gamma_{j}(S), \enspace i \in S_a, a \in \mathcal{A}.
\]

From the above,  (\ref{eq:fiaapS}), and (\ref{eq:giaapS}) we can define the average \emph{marginal (holding) cost metric}
\begin{equation}
\label{eq:afiaapS}
\bar{f}_i^{a, a'}(S) \triangleq  h_i^{a} - h_i^{a'} + 
\sum_{j \in \mathcal{N}} p_{ij}^{a} \varphi_j(S) - \sum_{j \in \mathcal{N}} p_{ij}^{a'} \varphi_j(S) = \lim_{\beta \nearrow 1} f_i^{a, a'}(S)
\end{equation}
and the the average \emph{marginal resource (usage) metric}
\begin{equation}
\label{eq:agiaapS}
\bar{g}_i^{a, a'}(S) \triangleq q_i^{a'} - q_i^{a} + 
\sum_{j \in \mathcal{N}} p_{ij}^{a'} \gamma_j(S) - \sum_{j \in \mathcal{N}} p_{ij}^a \gamma_j(S) = \lim_{\beta \nearrow 1} g_i^{a, a'}(S).
\end{equation}

If $\bar{g}_i^{a, a'}(S) > 0$, we further define the \emph{average MP metric}
\begin{equation}
\label{eq:alambdiaapS}
\overline{m}_i^{a, a'}(S) \triangleq \frac{\bar{f}_i^{a, a'}(S)}{\bar{g}_i^{a, a'}(S)} = \lim_{\beta \nearrow 1} m_i^{a, a'}(S).
\end{equation}

We thus have the following verification theorem, which is the average criterion counterpart to Theorem \ref{the:verthe} for the discounted criterion. Note that the following theorem refers to the corresponding concepts for the average criterion of $\mathcal{F}$-indexability, PCL$(\mathcal{F})$-indexability, and downshifting algorithm $\overline{\mathrm{DS}}(\mathcal{F})$, which is as algorithm $\mathrm{DS}(\mathcal{F})$ but using the average marginal metrics instead of the discounted ones.

\begin{theorem}
\label{the:averthe}
If the average cost model is PCL$(\mathcal{F})$-indexable, then it is  $\mathcal{F}$-indexable, with DAI  $\bar{\lambda}_j^{*, a}$ given by the MPI $\overline{m}_j^{*, a}$.
\end{theorem}

\subsection{Models with Uncontrollable States}
\label{s:mwuncst}
In the above framework, we have assumed that the DAI 
$\lambda_j^{*, a}$ is defined for all project states $j \in \mathcal{N}$. Yet, in some models, this need not be the case, in particular in those having \emph{uncontrollable states}. 
We call a project state $i$ \emph{uncontrollable} if only one action is available at $i$, or, equivalently, if all actions $a$ give the same transition probabilities, so 
$p_{ij}^a = p_{ij}^0$ for all $a$.
This concept was considered by the author in the corresponding framework for two-gear projects developed in~\cite{nmmp02}.
If there are uncontrollable states, we decompose the state space as $\mathcal{N} = \mathcal{N}^{\textup{cont}} \cup \mathcal{N}^{\textup{unc}}$, where $\mathcal{N}^{\textup{cont}}$ is the controllable state space and $\mathcal{N}^{\textup{unc}}$ is the uncontrollable state space.

In such a case, the above framework carries over by defining the concept of indexability and DAI by focusing on the controllable state space $\mathcal{N}^{\textup{cont}}$, so the DAI $\lambda_j^{*, a}$ will only be defined for states $j \in \mathcal{N}^{\textup{cont}}$. 
The required adaptions are straightforward. For example, the policy notation $S = (S_0, \ldots, S_A)$ used above can now be interpreted as meaning that, under such a policy, action $a$ is taken in controllable states $j \in S_a$ for $a = 0, \ldots, A$, as $S_0$, \ldots, $S_A$ is now a partition of $\mathcal{N}^{\textup{cont}}$. 

\subsection{Models with a Countably Infinite State Space}
\label{s:mcisspa}
The extension of the above framework to models with a countably infinite state space raises issues mainly in the definition of the downshift adaptive-greedy algorithm. Thus, the algorithm would not terminate, and it might not traverse the entire space of $(j, a)$ for which the index is defined. 
Furthermore, it might possibly entail choosing among infinitely many state--action pairs $(j, a)$ at each step.

Yet, in some countably infinite state space models such issues are easily addressed.
Consider, e.g., a model that might arise in queueing theory where the state is the number of customers in the system so the state space is the set of non-negative integers, $\mathcal{N} \triangleq \{0, 1, 2, \ldots\}$.
Imagine that the actions or gears $a$ correspond to server speeds, so higher gears give faster service rates.
The holding costs are used to model penalties (possibly nonlinear) for congestion.

In such a setting, it is natural to conjecture that optimal policies should be \emph{multi-threshold policies}. Any such policy is characterized by thresholds $z^1 \leqslant z^2 \leqslant \cdots \leqslant z^A$, with the interpretation that gear $0$ is used in states $1 \leqslant j \leqslant z^1$, gear $a$ is used in states $z^{a} < j \leqslant z^{a+1}$ for $a = 1, \ldots, A-1$, and gear $A$ is used in states $j > z^A$. Note that the optimality of such policies has been established in some queueing models, see, e.g.,~\cite{crabill72,sabeti73,ataShneorson06,mayorga06}.

The present framework would be applied to such a setting as follows. Rather than considering directly such multi-threshold policies and trying to establish their optimality, one would postulate a corresponding family of policies $\mathcal{F}$.
Note that in such a model state $0$ would be uncontrollable, as there is no meaningful choice of action when the queue is empty.
Thus, excluding state $0$ from consideration, the postulated family $\mathcal{F}$ would consist of partitions 
$S = (S_0, \ldots, S_A)$ of the controllable state space with the following property: if gear $a$ is selected in a state $j$ (i.e., $j \in S_a$), then at any lower state $j' < j$ a gear $a' \leqslant a$ must \mbox{be selected.}

It is easy to see that, in this setting, the natural extension of the downshift adaptive-greedy algorithm would indeed traverse the entire space of state--action pairs $(j, a)$ for which the DAI is defined, which would provide an alternative approach to address such problems to that previously considered  in the aforementioned literature. 

\section{Discussion}
\label{s:disc}
This paper has introduced novel sufficient conditions for the indexability of multi-gear bandits modeling a dynamic and stochastic project consuming a single resource, along with an efficient index-computing algorithm. 
This can be used to efficiently solve general MDP models that satisfy such conditions, and the index has further been used to design a heuristic index policy for the more complex multi-armed multi-gear bandit problem. 
This work opens a number of further avenues for developing such an approach, including the following: developing an efficient implementation of and testing the algorithm; 
deploying the new PCL-indexability conditions in a variety of relevant models arising in applications; extending the approach to models with a countable state space; and extending the approach to models with a continuous state space. 



\vspace{6pt} 




\section*{Funding}
{This research has been funded in part by the Spanish State Research Agency (\emph{Agencia Estatal de Investigaci\'on}, AEI) under grant PID2019-109196GB-I00 / AEI / 10.13039/501100011033 and by the \emph{Comunidad de Madrid} in the setting of the multi-year agreement with Carlos III University of Madrid within the line of activity ``\emph{Excelencia para el Profesorado Universitario}'', in the framework of the V Regional Plan of Scientific Research and Technological Innovation 2016--2020.}


\end{document}